\documentclass{amsart}
\usepackage{amsfonts,amssymb,amsmath,amsthm}
\usepackage{url}
\usepackage{enumerate}

\usepackage{mathrsfs}
\usepackage{bbold}

\newtheorem{thrm}{Theorem}[section]
\newtheorem{lem}[thrm]{Lemma}
\newtheorem{prop}[thrm]{Proposition}
\newtheorem{cor}[thrm]{Corollary}
\theoremstyle{definition}
\newtheorem{definition}[thrm]{Definition}

\numberwithin{equation}{section}

\newtheorem{claim}[thrm]{Claim}

\newcommand{\union}{\mathop{\bigcup}\limits}
\newcommand{\zfc}{\mathnormal{\mathsf{ZFC}}}
\newcommand{\inter}{\mathop{\bigcap}\limits}
\newcommand{\ch}{\mathnormal{\mathsf{CH}}}
\newcommand{\covm}{\mathnormal{\mathrm{cov}(\mathcal M)}}
\newcommand{\fs}{\mathop{\mathrm{FS}}}
\newcommand{\fu}{\mathop{\mathrm{FU}}}
\DeclareMathOperator{\supp}{supp}

\author[D. Fern\'andez]{David~J. Fern\'andez Bret\'on}
\address{
Department of Mathematics and Statistics\\
York University\\
Toronto, Ontario, Canada}
\email{davidfb@mathstat.yorku.ca}
\thanks{Partially supported by Conacyt (Mexico) and NSERC (Canada)}

\begin{document}

\title[Strongly Summable Ultrafilters]{Strongly summable ultrafilters, union ultrafilters, and the 
trivial sums property}

\begin{abstract}
We answer two questions of Hindman, Stepr\=ans and Strauss,\linebreak namely we prove that every 
strongly summable 
ultrafilter on an abelian group is sparse and has the trivial sums property. Moreover we 
show that in most 
cases the sparseness of the given ultrafilter is a 
consequence of its being isomorphic to a union ultrafilter. However, this does not happen 
in all cases: 
we also construct (assuming Martin's Axiom for countable partial orders, i.e. 
$\covm=\mathfrak c$), on the 
Boolean group, a strongly summable ultrafilter that 
is not additively isomorphic to any union ultrafilter. 
\end{abstract}
\maketitle

\section{Introduction}

The concept of a Strongly Summable Ultrafilter was born with Neil Hindman's 
efforts for proving the theorem that now bears his name (which at the time was 
known as Graham-Rothschild's conjecture), though later on it was realized that such ultrafilters 
have a rich algebraic structure in terms of the algebra in the \v Cech-Stone compactification, which in 
turn sheds light on the aforementioned theorem by providing an elegant proof of it. We 
conceive the \v Cech-Stone compactification of an abelian group $G$ (equipped with the discrete 
topology) as the set $\beta G$ of all ultrafilters on $G$, where the 
basic open sets are those of the form $\bar{A}=\{p\in\beta G\big|A\in p\}$, for $A\subseteq G$. As it 
turns out, these sets are actually clopen. If we identify each point $x\in G$ with the principal 
ultrafilter $\{A\subseteq G\big|x\in A\}$, then $G$ is a dense subset of $\beta G$, and what we 
denoted 
by $\bar{A}$ is really the closure in $\beta G$ of the set $A$. The group 
operation $+$ from $G$ is also 
extended by means of the formula
\begin{equation*}
 p+q=\{A\subseteq G\big|\{x\in G\big|A-x\in q\}\in p\}
\end{equation*}
which turns $\beta G$ into a right topological semigroup. This means that for each $p\in\beta G$, 
the mapping 
$(q\longmapsto q+p):\beta G\longrightarrow\beta G$ is continuous, although $\beta G$ is not a 
group (nonprincipal ultrafilters have no inverse). Moreover, the extended operation $+$ is not 
commutative in $\beta G$, even though its restriction to $G$ is; but elements $x\in G$ satisfy 
that 
$x+p=p+x$ for every $p\in\beta G$. The closed subsemigroup 
$G^*=\beta G\setminus G$ consisting of all nonprincipal ultrafilters will be of special 
importance. The book \cite{hindmanstrauss} is 
the standard reference on this topic.

We reserve the lowercase roman letters $p,q,r,u,v$ for ultrafilters, and the uppercase roman letters 
$A,B,C,D,W,X,Y,Z$, with or without subscripts, will always denote subsets of the abelian group at hand. 
Lowercase letters $w,x,y,z$ will typically denote elements of the abelian group that is being dealt 
with, and the ``vector'' notation will be used for sequences of elements of the group, e.g. 
$\vec{x}=\langle x_n\big|n<\omega\rangle$. When the sequences are finite, we use the symbol $\frown$ to 
denote their concatenation, as in $\vec{x}\frown\vec{y}$. If $G$ is an abelian group and $x\in G$, the 
symbol $o(x)$ will 
denote the \textit{order} of $x$, i.e. the least natural number $n$ such that $nx=0$.
We make liberal use of the von Neumann ordinals, usually denoted by Greek letters 
$\alpha,\beta,\gamma,\zeta,\eta,\xi$; thus for two ordinals $\alpha,\beta$, the expressions 
$\alpha<\beta$ and $\alpha\in\beta$ are interchangeable. In particular, a natural number $n$ is 
conceived as the set $\{0,\ldots,n-1\}$ 
of its predecessors, 
with $0$ being equal to the empty set $\varnothing$; and $\omega$ denotes the set of finite ordinals, 
i.e. the set $\mathbb N\cup\{0\}$. The lowercase roman letters $i,j,k,l,m,n$, with or without subscript, 
will 
be reserved to denote elements of $\omega$. The letters $M$ and $N$, with or without subscripts,
will in 
general be reserved for denoting subsets of $\omega$ (finite or 
infinite). Given a subset $M\subseteq\omega$, $[M]^n$ will denote the set of subsets 
of $M$ with $n$ elements, $[M]^{<\omega}=\union_{n<\omega}[M]^n$ will denote the set of 
finite subsets of $M$, and $[M]^{\omega}$ denotes the set of infinite subsets of $M$. The 
lowercase roman 
letters $a,b,c,d$, with or without subscript, will stand for 
elements of $[\omega]^{<\omega}$, i.e. for finite subsets 
of $\omega$.

Whenever we have a mapping $f:G\longrightarrow H$, there is a standard way to lift or extend it 
to another mapping $\beta f:\beta G\longrightarrow\beta H$ which is continuous and, if 
$f$ is a semigroup homomorphism, then so is $\beta f$. This extension is given by 
\begin{equation*}
 (\beta f)(p)=\{A\subseteq H\big|f^{-1}[A]\in p\}=\langle\{f[A]\big|A\in p\}\rangle,
\end{equation*}
where the rightmost expression means that we take the filter on $H$ generated by the 
family $\{f[A]\big|A\in p\}$, which has the finite intersection property. It 
is customary to write just $f(p)$ instead of $(\beta f)(p)$, and we will do so throughout 
this paper. The ultrafilter $f(p)$ is called the \textit{Rudin-Keisler image of $p$ under $f$}.

The cardinal invariant $\covm$ (read ``covering of meagre'') is the least cardinal for which 
Martin's 
Axiom fails at a countable partial order. This is, $\covm$ is the least $\kappa$ such that one can 
find $\kappa$-many dense subsets of some countable partial order with no filter meeting them all 
(this notation is explained by the fact that this cardinal is also the least possible number of 
meagre sets 
needed to cover all of the real line). Thus the equality $\covm=\mathfrak c$ means that Martin's 
Axiom holds for countable partial orders, whilst the failure of this principle is expressed by the 
inequality $\covm<\mathfrak c$.

One of the most 
important groups dealt with in this paper is the circle group $\mathbb T=\mathbb R/\mathbb Z$. When 
talking about this group, we will freely identify real numbers with their corresponding cosets 
modulo $\mathbb Z$, and conversely we will identify elements of $\mathbb T$ (cosets modulo 
$\mathbb Z$) with any of the elements of $\mathbb R$ representing them.  Therefore, when we 
refer to an element of $\mathbb T$ as a real number $t$, we really 
mean the coset of that number modulo $\mathbb Z$, thus e.g. we may write $t=0$ 
and really mean that $t\in\mathbb Z$. This should not cause confusion as the context 
will always clearly indicate whether we are viewing a real number $t$ as a real number or as 
an element of $\mathbb T$. If there is the need to specify a single representative for an element 
of $\mathbb T$, we will pick the unique representative $t$ 
satisfying $-\frac{1}{2}<t\leq\frac{1}{2}$.

We will now proceed to introduce the main objects of study of this paper.
\pagebreak
\begin{definition}
Let $G$ be an abelian group.
\begin{enumerate}[(i)]
 \item Given a $k$-sequence $\vec{x}=\langle x_i\big|i<k\rangle$ of elements of $G$ (where 
 $k\leq\omega$), we define the \textbf{set of finite sums of the 
 sequence $\vec{x}$} as:
 \begin{equation*}
  \fs(\vec{x})=\left\{\sum_{n\in a}x_n\bigg|a\in[k]^{<\omega}\setminus\{\varnothing\}\right\}.
 \end{equation*}
 \item An \textbf{FS-set} is just a set of the form $\fs(\vec{x})$ for some sequence $\vec{x}$ of 
 elements of $G$ with infinite range.
 \item An ultrafilter $p\in\beta G$ is \textbf{strongly summable} if it has 
a base of $\fs$-sets, i.e. if for every $A\in p$ 
there exists an $\omega$-sequence with infinite range, $\vec{x}=\langle x_n\big|n<\omega\rangle$, 
such that $p\ni\fs(\vec{x})\subseteq A$.
\end{enumerate}
\end{definition}

Note that the only principal strongly summable ultrafilter is $0$. 
Strongly summable ultrafilters on $(\mathbb N,+)$ were first constructed, under $\ch$, by 
Neil Hindman in \cite[Th. 3.3]{hindman} (here he claims to construct an idempotent, but a closer 
look at the proof reveals that the ultrafilter under construction is in fact strongly summable), 
although at that time 
this terminology was still not in use. The terminology was only introduced later on, in 
\cite[Def. 2.1]{hindman-obscure}. Blass and Hindman showed in \cite[Th. 3]{blasshindman} that the existence 
of strongly summable ultrafilters is not provable from the axioms of $\zfc$ alone, for it implies the 
existence of P-points. The sharpest result so far in terms of existence is due to Eisworth, 
who shows in \cite[Th. 9]{eisworth} that $\covm=\mathfrak c$ suffices for ensuring the existence of a 
strongly summable ultrafilter. In a forthcoming paper, this author shows that the existence of 
strongly summable ultrafilters on any abelian group is consistent with $\zfc$ together with 
$\covm<\mathfrak c$.

The importance of this type of ultrafilters came at first from the fact that 
they are examples of idempotents in $\beta\mathbb N$, but among idempotents they are 
special in that the largest subgroup of $\mathbb N^*$ 
containing one of them as the identity is just a copy of $\mathbb Z$. More concretely, 
\cite[Th. 12.42]{hindmanstrauss} establishes that if $p\in\mathbb N^*$ is a strongly 
summable ultrafilter, and $q,r\in\beta\mathbb N$ are such that $q+r=r+q=p$, 
then $q,r\in\mathbb Z+p$. In \cite{protasov}, the authors generalize some results 
previously only known to hold for ultrafilters on $\beta\mathbb N$ or $\beta\mathbb Z$. 
In particular, they proved there (\cite[Th. 2.3]{protasov}) that every strongly 
summable ultrafilter $p$ on any abelian group $G$ is an idempotent ultrafilter. And 
\cite[Th. 4.6]{protasov} states that if $G$ can be embedded 
in $\mathbb T$, then whenever $q,r\in G^*=\beta G\setminus G$ are such that $q+r=r+q=p$, 
it must be the 
case that $q,r\in G+p$. The following definition captures an even stronger 
property than the one just mentioned.

\begin{definition}
 If $p\in\beta G$ is an idempotent element, we say that $p$ has the \textbf{trivial 
 sums property} if whenever $q,r\in\beta G$ are such that $q+r=p$, then it must 
 be the case that $q,r\in G+p$.
\end{definition}

Note that $0$ always has the trivial sums property, because $G^*$ is an ideal 
of $\beta G$. Idempotents satisfying the trivial sums property 
would be examples of so-called \textit{maximal idempotents}, i.e., maximal 
elements with respect to the two partial orders $\leq_R,\leq_L$ defined 
among idempotents by $q\leq_R r$ iff $r+q=q$ and $q\leq_L r$ iff $q+r=q$. 
It is possible to improve the result just mentioned for strongly summable 
ultrafilters if one strengthens the definition of strongly summable.

\begin{definition}
An ultrafilter $p\in\beta G$ is \textbf{sparse} if for every $A\in p$ there exist two sequences 
$\vec{x}=\langle x_n\big|n<\omega\rangle$, $\vec{y}=\langle y_n\big|n<\omega\rangle$, where 
$\vec{y}$ is a subsequence 
of $\vec{x}$ such that $\{x_n\big|n<\omega\}\setminus\{y_n\big|n<\omega\}$ is infinite, 
$\fs(\vec{x})\subseteq A$, and 
$\fs(\vec{y})\in p$.
\end{definition}

Then obviously every sparse ultrafilter will be nonprincipal and strongly summable. And by
\cite[Th. 4.5]{protasov}, if $G$ can be embedded in $\mathbb T$ and $p\in G^*$ is sparse, 
then $p$ has the trivial sums property. In some non-commutative settings (adapting the 
relevant definitions appropriately), the relationship between sparseness and an analogue 
of the trivial 
sums property has been further explored (see~\cite{lakeshia}).

It follows from results of Krautzberger (\cite[Props. 4 and 5, and Th. 4]{krautzberger}) that every 
nonprincipal 
strongly 
summable ultrafilter $p\in\mathbb N^*$ must actually be sparse. Thus the previous theorem 
holds for nonprincipal strongly summable ultrafilters on $\mathbb N$, i.e. every such 
ultrafilter, being sparse, has the trivial sums property. In \cite{jurisetal}, 
the authors followed this idea and started investigating the different kinds of abelian 
semigroups on which every nonprincipal strongly summable ultrafilter must be sparse. In 
particular, \cite[Th. 4.2]{jurisetal} establishes that if $S$ is a countable subsemigroup of 
$\mathbb T$, then every nonprincipal strongly summable ultrafilter on $S$ is sparse, so 
this generalizes the previous observation about strongly summable ultrafilters on 
$\mathbb N$. The authors built on this result to get a more general result 
(\cite[Th. 4.5 and Cor. 4.6]{jurisetal}) outlining a large class of abelian groups, whose 
nonprincipal strongly summable ultrafilters must all be sparse. More or less concurrently, 
this author showed (\cite[Th. 2.1]{yonilaboriel}) that every nonprincipal strongly summable 
ultrafilter on the Boolean group is also sparse. 
Thus Hindman, Stepr\=ans and Strauss asked (\cite[Question 4.12]{jurisetal}) whether 
every strongly summable ultrafilter on a countable abelian group is sparse.

Although it is not immediately clear that, for groups that are not embeddable in 
$\mathbb T$, sparseness implies the trivial sums property, Hindman, Stepr\=ans 
and Strauss were able to get a result, analogous to the ones mentioned in the 
previous paragraph, concerning the latter property, namely they proved 
(\cite[Th. 4.8 and Cor. 4.9]{jurisetal}) that for the same class of abelian 
groups, all nonprincipal strongly summable ultrafilters must have the trivial 
sums property. The analogous result for the Boolean group had already been 
proved, long time ago, by Protasov (\cite[Cor. 4.4]{protasums}). Thus Hindman, 
Stepr\=ans and Strauss also asked (\cite[Question 4.11]{jurisetal}) whether 
every strongly summable ultrafilter on a countable abelian group $G$ has the 
property that it can only be expressed trivially as a product (i.e. a sum) in 
$G^*$.

Section 2 develops some preliminary results that deal with union 
ultrafilters, 
additive isomorphisms and what we call here the 2-uniqueness of finite 
sums. Section 3 contains 
the answer to the two questions from \cite{jurisetal} mentioned in the 
previous paragraphs. From the proof of this result, it will turn out that, unless $p$ is a strongly 
summable ultrafilter on the Boolean group, it will be additively 
isomorphic to a union ultrafilter. Thus Section 4 deals with the 
Boolean group, the main result being that, under the assumption that $\covm=\mathfrak c$ 
(this is, under Martin's Axiom for countable forcing notions), there exists 
a strongly summable 
ultrafilter on the Boolean group that is not additively isomorphic 
to any union ultrafilter.

\section{Union ultrafilters and 2-uniqueness of finite sums}

Union ultrafilters were first defined by Blass in \cite[p. 92]{blass}, an article that 
appeared in the same 
volume as that of Hindman's (\cite{hindman-obscure}) where strongly summable ultrafilters 
are first defined. So ever since their inception, the notions of union 
ultrafilter and of strongly summable ultrafilter have always been inextricably 
related. The results of this paper are no exception, and the notion of union ultrafilter 
is essential to 
them. We thus introduce such notion. For a pairwise disjoint family 
$X\subseteq[\omega]^{<\omega}$, we 
denote the set of its \textbf{finite unions} by
\begin{equation*}
 \fu(X)=\left\{\union_{x\in a}x\bigg|a\in[X]^{<\omega}\setminus\{\varnothing\}\right\}.
\end{equation*}

\begin{definition}
A \textbf{union ultrafilter} is an ultrafilter $p$ on $[\omega]^{<\omega}$ such that for every $A\in p$ it is 
possible to find a pairwise disjoint $X\subseteq[\omega]^{<\omega}$ such that $p\ni\fu(X)\subseteq A$.
\end{definition}

The reason why union ultrafilters are so important when studying strongly summable ultrafilters, is that 
sometimes strongly summable ultrafilters can be used to construct union ultrafilters, which in turn 
are sometimes easier to handle. We will state a definition that captures the precise sense in which 
strongly 
summable ultrafilters give rise to union ultrafilters. In order to do this, we need to 
introduce a further 
notion, which stems from the fact that, when dealing with sets of the form $\fs(\vec{x})$, 
if each finite sum from this set
can be expressed 
uniquely 
as such then the situation is much more comfortable. To simplify notation, we make the convention 
that 
for any sequence $\vec{x}$ of elements of some abelian group $G$, the \textbf{empty sum} equals 
zero:
\begin{equation*}
\sum_{n\in\varnothing}x_n=0. 
\end{equation*}

\begin{definition}
 A sequence $\vec{x}$ on an abelian group $G$ is said to satisfy \textbf{uniqueness of finite sums} if 
 whenever $a,b\in[\omega]^{<\omega}$ are such that
 \begin{equation*}
  \sum_{n\in a}x_n=\sum_{n\in b}x_n,
 \end{equation*}
 it must be the case that $a=b$.
\end{definition}

In particular, if $\vec{x}$ satisfies uniqueness of finite sums then $0\notin\fs(\vec{x})$. Now we are 
ready to introduce the notion that will provide the connection between strongly summable 
ultrafilters and union ultrafilters.

\begin{definition}
Let $p$ be an ultrafilter on an abelian group $G$, and let $q$ be a union ultrafilter. We say that $p$ and 
$q$ are \textbf{additively isomorphic} if there is a sequence $\vec{x}$ of elements of $G$ satisfying 
uniqueness of finite sums, such that 
$\fs(\vec{x})\in p$, and 
there is a pairwise disjoint family $Y=\{y_n\big|n<\omega\}$ of elements of $[\omega]^{<\omega}$, 
in such a way that the mapping 
$\varphi:\fs(\vec{x})\longmapsto\fu(Y)$ given by $\varphi(\sum_{n\in a}x_n)=\union_{n\in a}y_n$ 
maps $p$ to $q$.
\end{definition}

If we are only interested in determining whether a given strongly summable ultrafilter 
$p$ is additively isomorphic to \textit{some} union ultrafilter, without worrying about which 
this ultrafilter is exactly, then we can assume without loss of generality that the isomorphism 
is fairly simple. This is established formally and precisely in the following proposition.

\begin{prop}\label{addisomissimple}
If $p$ is additively isomorphic to a union ultrafilter, and this is witnessed by the mapping 
$\sum_{n\in a}x_n\longmapsto\union_{n\in a}y_n$ from $\fs(\vec{x})$ to $\fu(Y)$, then the mapping 
$\psi:\fs(\vec{x})\longrightarrow[\omega]^{<\omega}$ given by $\psi(\sum_{n\in a}x_n)=a$ also maps $p$ to 
a union ultrafilter.
\end{prop}

\begin{proof}
We only need to show that for any union ultrafilter $q$ and any pairwise disjoint 
$Y=\{y_n\big|n<\omega\}$ 
such that $\fu(Y)\in q$, the mapping $\varphi$ given by $\union_{n\in a}y_n\longmapsto a$
maps $q$ to another 
union ultrafilter. Once we prove this, then given the hypothesis of the theorem we can just 
compose the mapping
$\varphi$ with the original isomorphism to get the $\psi$ that we need. So let 
$r$ be the image of $q$ under such mapping, and let 
$A\in r$. Then since $B=\varphi^{-1}[A]\in q$, there is a pairwise disjoint $X$ 
such that $q\ni\fu(X)\subseteq B\cap\fu(Y)$. Since $X$ is pairwise disjoint and contained 
in $\fu(Y)$, it is readily checked that for distinct $x,w\in X$, if 
$x=\union_{n\in a}y_n$ 
and $w=\union_{n\in b}y_n$ then $a\cap b=\varnothing$. Hence the family 
$Z=\{a\in[\omega]^{<\omega}\big|\union_{n\in a}y_n\in X\}$ is pairwise disjoint.
Note moreover that all finite unions are preserved in the sense that, for 
$x_0,\ldots,x_n\in X$ such that $x_i=\union_{k\in a_i}y_k$, we have that  
$\union_{i=0}^n x_i=\union_{k\in a}y_k$ where $a=\union_{i=0}^n a_i$, i.e. 
$\varphi\left(\union_{i=0}^n x_i\right)=\union_{i=0}^n \varphi(x_i)$. This means 
that 
$\varphi[\fu(X)]=\fu(Z)$,
thus $r\ni\fu(Z)\subseteq A$ and we are done.
\end{proof}

We will develop a useful criterion for knowing when a strongly summable ultrafilter 
is additively isomorphic 
to some union ultrafilter. For that, it will be useful to think of the uniqueness 
of finite sums as a 
\textit{1-uniqueness of finite 
sums}, in 
the sense that the expressions under consideration only have coefficients equal to 1. 
With this in mind, 
it is natural 
to try and define a corresponding 2-uniqueness where we allow coefficients 1 and 2. More formally,

\begin{definition}
 A sequence $\vec{x}$ on an abelian group $G$ is said to satisfy the\linebreak 
 \textbf{2-uniqueness of finite sums}
 if whenever $a,b\in[\omega]^{<\omega}$ and 
 $\varepsilon:a\longrightarrow\{1,2\},\delta:b\longrightarrow\{1,2\}$ are such that
 \begin{equation*}
  \sum_{n\in a}\varepsilon(n)x_n=\sum_{n\in b}\delta(n)x_n,
 \end{equation*}
 it must be the case that $a=b$ and $\varepsilon=\delta$.
\end{definition}

In particular, if $\vec{x}$ satisfies 2-uniqueness of finite sums, then no element of $\fs(\vec{x})$ 
can have order 2. Thus Boolean groups do not contain sequences satisfying 2-uniqueness of finite 
sums. It is of course possible to analogously define $n$-uniqueness of finite sums, for every $n$, but 
for the results of this paper we only need to consider the case $n=2$.

\begin{prop}\label{carac2uniq}
 For a sequence $\vec{x}$ on an abelian group $G$, the following are equivalent.
 \begin{enumerate}[(i)]
  \item $\vec{x}$ satisfies the 2-uniqueness of finite sums.
  \item Whenever $a,b,c,d\in[\omega]^{<\omega}$ are such that $a\cap b=\varnothing=c\cap d$, if
  \begin{equation*}
   2\sum_{n\in a}x_n+\sum_{n\in b}x_n=2\sum_{n\in c}x_n+\sum_{n\in d}x_n
  \end{equation*}
  then $a=c$ and $b=d$.
  \item Whenever $a,b,c,d\in[\omega]^{<\omega}$ are such that
  \begin{equation*}
   \sum_{n\in a}x_n+\sum_{n\in b}x_n=\sum_{n\in c}x_n+\sum_{n\in d}x_n,
  \end{equation*}
  it must be the case that $a\bigtriangleup b=c\bigtriangleup d$ and $a\cap b=c\cap d$.
 \end{enumerate}
\end{prop}

\begin{proof}
 Straightforward.
\end{proof}

The following two theorems do not contain any new ideas but rather they are just a useful reformulation of  
\cite[Th. 3.2]{jurisetal} (although that theorem uses a condition that is slightly weaker than the 
2-uniqueness of finite sums, namely what the authors call the ``strong uniqueness of finite sums''; however 
the version that we present here will be enough for our purposes) that cuts it into two pieces, each of 
which will be of some use in the future.
Besides, we think that the distinction made here is illuminating.

\begin{thrm}\label{addisomtounion}
Let $p$ be a strongly summable ultrafilter such that for some $\vec{x}$ satisfying 2-uniqueness of finite sums, 
$\fs(\vec{x})\in p$. Then $p$ is additively isomorphic to a union ultrafilter.
\end{thrm}

\begin{proof}
We just need to check that the mapping $\varphi$ given by 
$\varphi(\sum_{n\in a}x_n)=a$ sends $p$ to a union 
ultrafilter. So let $A\in q=\varphi(p)$. Pick a sequence $\vec{y}$ such that 
$p\ni\fs(\vec{y})\subseteq\varphi^{-1}[A]$. Then $\varphi[\fs(\vec{y})]\subseteq A$. Now 
$\varphi^{-1}[A]\subseteq\fs(\vec{x})$, thus for each 
$n<\omega$ we can define $c_n\in[\omega]^{<\omega}$ by $c_n=\varphi(y_n)$ or, equivalently, by 
$y_n=\sum_{i\in c_n}x_i$. We claim that the family 
$C=\{c_n\big|n<\omega\}$ is 
pairwise disjoint. This is because if $n\neq m$, since $y_n+y_m\in\fs(\vec{y})\subseteq\fs(\vec{x})$, then 
there must be a $c\in[\omega]^{<\omega}$ such that 
\begin{equation*}
 \sum_{i\in c}x_i=y_n+y_m=\sum_{i\in c_n}x_i+\sum_{i\in c_m}x_i.
\end{equation*}
Since $\vec{x}$ satisfies 2-uniqueness of finite sums, by Proposition~\ref{carac2uniq} 
we can conclude that $c=c_n\cup c_m$ and 
$c_n\cap c_m=\varnothing$.
This argument shows at once that $C$ is a pairwise disjoint family, and that 
$\varphi(y_n+y_m)=c_n\cup c_m=\varphi(y_n)\cup\varphi(y_m)$. 
From this it is easy to 
prove by induction that $\varphi\left(\sum_{n\in a}y_n\right)=\union_{n\in a}\varphi(y_n)$, 
for all $a\in[\omega]^{<\omega}$, hence
$\varphi[\fs(\vec{y})]=\fu(C)$, therefore $q\ni\fu(C)\subseteq A$ and we are done.
\end{proof}


\begin{thrm}
Let $p$ be an ultrafilter that is additively isomorphic to a union ultrafilter. Then $p$ is sparse.
\end{thrm}

\begin{proof}
If $p$ is additively isomorphic to some union ultrafilter, by Proposition~\ref{addisomissimple} we 
can pick a sequence $\vec{x}$ satisfying uniqueness of finite sums such that $\fs(\vec{x})\in p$, and 
such that the mapping 
$\varphi$ given by $\varphi(\sum_{n\in a}x_n)=a$ maps $p$ to a union ultrafilter $q$. Let $A\in p$, and 
let $X$ be pairwise disjoint such that $q\ni\fu(X)\subseteq\varphi[A\cap\fs(\vec{x})]$. Now let $M=\union X$. 
Since $q$ is a union ultrafilter, \cite[Th. 4]{krautzberger} (cf. also \cite[Th. 2.6]{jurisetal})
ensures that
there is $B\in q$ such that $M\setminus\union B$ is infinite. 
Without loss of generality we can assume $B\subseteq\fu(X)$, so that $\union B$ is a coinfinite subset of 
$M$. Grab a pairwise disjoint family $Y$ such that 
$q\ni\fu(Y)\subseteq B$, then $\union Y$ is a coinfinite subset of $M=\union X$ and thus there are infinitely 
many $x\in X$ that do not intersect $\union Y$ (because $Y\subseteq\fu(X)$ and 
$X$ is a pairwise disjoint family, so if $x\in X$ intersects $\union Y$ then $x\subseteq\union Y$). Thus if 
we let 
$Z=\{x\in X\big|x\cap\union Y=\varnothing\}\cup Y$ then $Z$ is a pairwise disjoint family and 
$\fu(Z)\subseteq\fu(X)\subseteq\varphi[A\cap\fs(\vec{x})]$. 
Enumerate $Z=\{z_n\big|n<\omega\}$ in such a way that $Y=\{z_{2n}\big|n<\omega\}$ and 
$\{x\in X\big|x\cap\union Y=\varnothing\}=\{z_{2n+1}\big|n<\omega\}$. Then let $\vec{w}$ be given by 
$w_n=\sum_{i\in z_n}x_i$. We get that $\fs(\vec{w})=\varphi^{-1}[\fu(Z)]\subseteq A$, and if 
$\vec{y}$ is
the subsequence of 
even elements of $\vec{w}$, then we will have that 
$|\{w_n\big|n<\omega\}\setminus\{y_n\big|n<\omega\}|$ is 
infinite and $\fs(\vec{y})=\varphi^{-1}[\fu(Y)]\in p$.
\end{proof}

\begin{cor}[\cite{jurisetal}, Th. 3.2.]\label{eldelesparcido}
 Let $p$ be a strongly summable ultrafilter on some abelian group $G$ such that there exists a sequence 
 $\vec{x}$ satisfying the 2-uniqueness of finite sums with $\fs(\vec{x})\in p$. Then $p$ is sparse.\hfill$\Box$
\end{cor}

To finish this section, we would like to quote another result from \cite{jurisetal} that will be relevant 
in the subsequent section, and that illustrates another application of the concept of 2-uniqueness of 
finite sums.

\begin{thrm}[\cite{jurisetal}, Th. 4.8]\label{eldelassumas}
 Let $G$ be an abelian group, and $p\in G^*$ be a strongly summable ultrafilter such that there exists a 
 sequence $\vec{x}$ satisfying the 2-uniqueness of finite sums, with $\fs(\vec{x})\in p$. Then $p$ 
 has the trivial sums property.
\end{thrm}

\section{Strongly Summable Ultrafilters are Sparse and have the Trivial Sums Property}

The main result of this section tells us that almost all strongly summable ultrafilters on abelian 
groups have $\fs$-sets generated from sequences that satisfy 2-uniqueness of finite sums. 
As a consequence 
of that, almost all strongly summable ultrafilters on abelian groups 
are essentially union ultrafilters (because of Theorem~\ref{addisomtounion}), and 
this helps solve \cite[Questions 4.11 and 4.12]{jurisetal}. More precisely, we have the following theorem 
and corollary.

\begin{thrm}\label{mainresult}
 Let $G$ be an abelian group, and let $p\in G^*$ be a strongly summable ultrafilter such that 
 \begin{equation*}
  \{x\in G\big|o(x)=2\}\notin p.
 \end{equation*}
 Then, there exists a sequence $\vec{x}$ of elements of $G$ satisfying the 2-uniqueness of 
 finite sums such 
 that $\fs(\vec{x})\in p$.
\end{thrm}

\begin{cor}\label{nonbooleanimpliesunion}
 Let $G$ be an abelian group, and let $p\in G^*$ be a strongly summable ultrafilter such that 
 \begin{equation*}
  \{x\in G\big|o(x)=2\}\notin p.
 \end{equation*}
 Then $p$ is additively isomorphic to some union ultrafilter.
\end{cor}

In order to prove this result, we will need to break the proof down into several subcases.


\begin{lem}\label{order4}
 Let $G$ be an abelian group, and let $X=\{x\in G\big|o(x)=4\}$. If $\vec{x}$ is a sequence 
 of elements of $G$ such that $\fs(\vec{x})\subseteq X$, then $\vec{x}$ must satisfy 2-uniqueness 
 of finite sums.
\end{lem}

\begin{proof}
 Assume that $\vec{x}$ is such that $\fs(\vec{x})\subseteq X$. By Proposition~\ref{carac2uniq}, 
 in order to prove that $\vec{x}$ satisfies 2-uniqueness of finite sums, it suffices to show that 
 whenever $a,b,c,d$ are such that $a\cap b=\varnothing=c\cap d$ and 
 \begin{equation*}
  2\sum_{n\in a}x_n+\sum_{n\in b}x_n=2\sum_{n\in c}x_n+\sum_{n\in d}x_n,
 \end{equation*}
 then $a=c$ and $b=d$. Now for each $n\in b\cap d$ we can cancel the term $x_n$ from both sides 
 of the previous equation; and similarly for each $n\in a\cap c$ we can cancel the term 
 $2x_n$ from both sides of the equation, which thus becomes 
 \begin{equation}\label{strongu1}
 2\sum_{n\in a'}x_n+\sum_{n\in b'}x_n=2\sum_{n\in c'}x_n+\sum_{n\in d'}x_n,
 \end{equation}
 where $a'=a\setminus(a\cap c)$, $b'=b\setminus(b\cap d)$, $c'=c\setminus(a\cap c)$ 
 and $d'=d\setminus(b\cap d)$. Since $b'$ is disjoint 
 from $d'$, Equation~(\ref{strongu1}) yields 
 \begin{equation*}
  \sum_{n\in b'\cup d'}x_n=\sum_{n\in b'}x_n+\sum_{n\in d'}x_n=-2\sum_{n\in a'}x_n+2\sum_{n\in c'}x_n+2\sum_{n\in d'}x_n,
 \end{equation*}
 where the right-hand side is either the identity or has order $2$, while the left-hand sideis either the 
 identity or has order $4$. Hence both sides of this equation must be the identity, and so $b'\cup d'=\varnothing$, 
 this is, $b'=d'=\varnothing$ and hence $b=b\cap d=d$. Therefore (\ref{strongu1}) becomes 
 \begin{equation*}
  2\sum_{n\in a'}x_n=2\sum_{n\in c'}x_n,
 \end{equation*}
 which in turn implies that 
 \begin{equation*}
  2\sum_{n\in a'\cup c'}x_n=4\sum_{n\in c'}x_n=0,
 \end{equation*}
 and this can only happen if $a'\cup c'=\varnothing$, which means 
 that $a'=c'=\varnothing$ and hence $a=a\cap c=d$. So we have that 
 $\vec{x}$ satisfies 2-uniqueness 
 of finite sums.
\end{proof}

If $G$ is any abelian group, and $p\in G^*$ is strongly summable, then there must be 
a countable subgroup $H$ such that $H\in p$ (e.g. take any $\fs$ set in $p$ because 
of strong summability, and then let $H$ be the subgroup generated by such $\fs$ set), 
and certainly the restricted ultrafilter $p\upharpoonright H=p\cap\mathfrak P(H)$ 
will also be strongly summable. If we prove that $p\upharpoonright H$ contains a set 
of the form $\fs(\vec{x})$ for a sequence $\vec{x}$ satisfying 2-uniqueness of finite 
sums, then certainly so does $p$ itself, because $p$ is just the ultrafilter generated 
in $G$ by $p\upharpoonright H$ and in particular $p\upharpoonright H\subseteq p$. 
Hence in order to prove Theorem~\ref{mainresult}, 
it suffices to consider only countable abelian groups $G$, and we will do so in the 
remainder of this section.

Now, it is a well-known result (this is mentioned in \cite[p. 123, Sect. 1]{protasov}, 
and thoroughly 
discussed at the beginning of 
\cite[Section 3]{yonilaboriel}) that every countable abelian group $G$ can be 
embedded 
in a countable direct sum of circle groups $\bigoplus_{n<\omega}\mathbb T$. Thus 
from now on we will use this fact liberally, in particular all elements $x$ of the 
abelian group under consideration will be thought of as $\omega$-sequences, each 
of whose terms is an element of $\mathbb T$. We will denote by $\pi_n$ the 
projection map onto the $n$-th. coordinate, i.e. $\pi_n(x)$ is the $n$th. term 
of the sequence that $x$ represents.

\begin{definition}
 When dealing with an arbitrary (countable) abelian group $G$, we will denote by 
 $Q(G)=\{x\in G\big|o(x)>4\}$. Since elements of $G$ are elements of 
 $\bigoplus_{n<\omega}\mathbb T$, if $x\in Q(G)$ then there is an $n<\omega$ such 
 that $\pi_n(x)\notin\left\{0,\frac{1}{4},-\frac{1}{4},\frac{1}{2}\right\}$. We
 will denote the least such $n$ by $\rho(x)$.
\end{definition}

At this point it is worth recalling the following theorem of Hindman, Stepr\=ans 
and Strauss.

\begin{thrm}[\cite{jurisetal}, Th. 4.5]\label{theoremjuris}
Let $S$ be a countable subsemigroup of $\bigoplus_{n<\omega}\mathbb T$, and let $p$ be a 
nonprincipal strongly summable ultrafilter on $S$. If
\begin{equation*}
 \left\{x\in S\big|\pi_{\min(x)}(x)\neq\frac{1}{2}\right\}\in p,
\end{equation*}
where $\min(x)$ denotes the least $n$ such that $\pi_n(x)\neq0$, then there exists a 
set $X\in p$
such that for every sequence $\vec{x}$ of elements of $\bigoplus_{n<\omega}\mathbb T$, if 
$\fs(\vec{x})\subseteq X$ then $\vec{x}$ must satisfy 2-uniqueness of finite sums.
\end{thrm}

This theorem is the tool which will allow us to prove the following lemma.

\begin{lem}\label{eluncuarto}
 Let $G$ be an abelian group, and let $p\in G^*$ be a strongly summable ultrafilter. If
 \begin{equation*}
  \left\{x\in Q(G)\bigg|\pi_{\rho(x)}(x)\notin\left\{\frac{1}{8},-\frac{1}{8},\frac{3}{8},-\frac{3}{8}\right\}\right\}\in p,
 \end{equation*}
then there exists a set $X\in p$
such that for every sequence $\vec{x}$ of elements of $\bigoplus_{n<\omega}\mathbb T$, if 
$\fs(\vec{x})\subseteq X$ then $\vec{x}$ must satisfy 2-uniqueness of finite sums.
\end{lem}

\begin{proof}
Consider the morphism $\varphi:G\longrightarrow G\subseteq\bigoplus_{n<\omega}\mathbb T$ 
given by $\varphi(x)=4x$, whose kernel is exactly $G\setminus Q(G)$. Since the latter is 
not an element of $p$, then $\varphi(p)$ is a nonprincipal ultrafilter. Moreover, since 
$p$ is strongly summable, so is $\varphi(p)$ by \cite[Lemma 4.4]{jurisetal}. Now notice 
that for $x\in G\setminus\ker(\varphi)=Q(G)$, we have $\rho(x)=\min(\varphi(x))$. Thus 
$\varphi(p)$ contains the set $\{x\in G\setminus\{0\}\big|\pi_{\min(x)}(x)\neq1/2\}$, 
since its preimage under 
$\varphi$ is exactly 
$\left\{x\in Q(G)\bigg|\pi_{\rho(x)}(x)\notin\left\{\frac{1}{8},-\frac{1}{8},\frac{3}{8},-\frac{3}{8}\right\}\right\}$. 
Therefore by Theorem~\ref{theoremjuris}, there is a 
set $Y\in\varphi(p)$ such that whenever $\fs(\vec{y})\subseteq Y$, $\vec{y}$ must 
satisfy 2-uniqueness of finite sums. If we let $X=\varphi^{-1}[Y]$, we claim that $X\in p$ 
is the set that we need. So let $\vec{x}$ be a sequence such that 
$\fs(\vec{x})\subseteq X$. Then letting $\vec{y}$ be the sequence given by 
$y_n=\varphi(x_n)$, since $\varphi$ is a group homomorphism we get that 
$\fs(\vec{y})=\varphi[\fs(\vec{x})]\subseteq\varphi[X]\subseteq Y$, 
thus $\vec{y}$ must satisfy 2-uniqueness of finite sums. 
Again since $\varphi$ is a group homomorphism, it is not hard to see that this implies 
that $\vec{x}$ satisfies 2-uniqueness of finite sums as well, and we are done.
\end{proof}

The following theorem is the last piece needed for proving Theorem~\ref{mainresult}.

\begin{thrm}\label{micontri}
 Let $G$ be an abelian group, and let $p\in G^*$ be a strongly summable ultrafilter. If 
 \begin{equation*}
  \left\{x\in Q(G)\bigg|\pi_{\rho(x)}(x)\in\left\{\frac{1}{8},-\frac{1}{8},\frac{3}{8},-\frac{3}{8}\right\}\right\}\in p,
 \end{equation*}
 then there exists a set $X\in p$
 such that for every sequence $\vec{x}$ of elements of $\bigoplus_{n<\omega}\mathbb T$, if 
 $\fs(\vec{x})\subseteq X$ then $\vec{x}$ must satisfy 2-uniqueness of finite sums.
\end{thrm}

\begin{proof}
If $p\in G^*$ is as described in the hypothesis, then there is an $i\in\{1,-1,3,-3\}$
such that 
\begin{equation*}
 Q_i=\left\{x\in Q(G)\bigg|\pi_{\rho(x)}(x)=\frac{i}{8}\right\}\in p.
\end{equation*}
Let $\vec{x}$ be such that $p\ni\fs(\vec{x})\subseteq Q_i$. For 
$j<\omega$ let $M_j=\{n<\omega\big|\rho(x_n)=j\}$.

\pagebreak

\begin{claim}\label{atmosttwo}
For each $j<\omega$, $|M_j|\leq2$.
\end{claim}

\begin{proof}[Proof of Claim]
Assume, by way of contradiction, that there are three distinct $n,m,k\in M_j$, 
and let $x=x_n+x_m+x_k$. For $l<j$, $\pi_l(x)$ must be an element of 
$\left\{0,\frac{1}{4},-\frac{1}{4},\frac{1}{2}\right\}$, because so are 
$\pi_l(x_n),\pi_l(x_m)$ and $\pi_l(x_k)$. On the other hand, 
$\pi_j(x_n)=\pi_j(x_m)=\pi_j(x_k)=\frac{i}{8}$, so $\rho(x)=j$ but 
$\pi_j(x)=\frac{3i}{8}\neq\frac{i}{8}$.
\end{proof}

Thus we can rearrange the sequence $\vec{x}$ in such a way that $n<m$ implies 
$\rho(x_n)\leq\rho(x_m)$, where the inequality is strict if $m>n+1$. Let 
$M=\{\rho(x_n)\big|n<\omega\}$.

\begin{claim}\label{iszeroonpreviousrho}
Let $n<m<\omega$ and assume that $j=\rho(x_n)<\rho(x_m)$ (which may or may 
not hold if $m=n+1$, but must hold if 
$m>n+1$). Then $\pi_j(x_m)=0$.
\end{claim}

\begin{proof}[Proof of Claim]
Let $x=x_n+x_m$. Arguing as in the proof of Claim~\ref{atmosttwo}, we get that $\rho(x)=j$ 
and thus since $x\in Q_i$, $\pi_j(x_n)+\pi_j(x_m)=\pi_j(x)=\frac{i}{8}$. 
Now on the one hand we know that 
$\pi_j(x_m)\in\left\{0,\frac{1}{4},-\frac{1}{4},\frac{1}{2}\right\}$, 
while on the other hand $\pi_j(x_n)=\frac{i}{8}$. Hence the only possibility that does not lead 
to contradiction is that $\pi_j(x_m)=0$.
\end{proof}

\begin{claim}
For every $x\in\fs(\vec{x})$ there is a $j\in M$ such that $\pi_j(x)\neq0$. Moreover for the least such $j$ 
we actually have that $\pi_j(x)\in\left\{\frac{i}{8},\frac{2i}{8}\right\}$.
\end{claim}

\begin{proof}[Proof of Claim]
For if $x=\sum_{n\in a}x_n$ and if $m=\min(a)$, then we can let 
$j=\rho(x_m)\in M$, so that for every $n\in a$ we have 
$\rho(x_n)\geq j$, with a strict inequality if $n>m+1$.
Now, we have that   
\begin{equation*}
 \pi_j(x)=\sum_{n\in a}\pi_j(x),
\end{equation*}
where, by Claim~\ref{iszeroonpreviousrho}, each of the terms 
on the right-hand side of this expression are zero, except 
for $\pi_j(x_m)=\frac{1}{8}$ and possibly $\pi_j(x_{m+1})$
(which will appear on the summation only if $m+1\in a$, and 
if so it will equal $\frac{1}{8}$ if $\rho(x_{m+1})=\rho(x_m)$, 
and zero otherwise). Thus 
$\pi_j(x)\in\left\{\frac{i}{8},\frac{2i}{8}\right\}$. In particular 
$\pi_j(x)\neq0$, now in order to prove the ``moreover'' part, we will 
argue that for all $l<j$ such that $l\in M$, $\pi_l(x)=0$. This is 
because if $l\in M$, then there is $k<\omega$ such that 
$\rho(x_k)=l$, and if $l<j$ then we must necessarily have $k<m$ 
because of the way we arranged our sequence $\vec{x}$. Hence, again 
by Claim~\ref{iszeroonpreviousrho} and since $m=\min(a)$, it will be the case that 
$\pi_l(x_n)=0$ for all $n\in a$, and hence 
\begin{equation*}
 \pi_l(x)=\sum_{n\in a}\pi_l(x_n)=0,
\end{equation*}
therefore $j$ is actually the least $l\in M$ such that $\pi_l(x)\neq0$ 
and we are done.
\end{proof}

The previous claim allows us to define $\tau:\fs(\vec{x})\longmapsto M$ 
by $\tau(x)=\min\{j\in M\big|\pi_j(x)\neq 0\}$, and ensures that 
$\pi_{\tau(x)}(x)\in\left\{\frac{i}{8},\frac{2i}{8}\right\}$. We 
can thus let 
\begin{equation*}
 C_k=\left\{x\in\fs(\vec{x})\bigg|\pi_{\tau(x)}(x)=\frac{ki}{8}\right\}
\end{equation*}
for $k\in\{1,2\}$, and choose from among those the $k$ such that 
$C_k\in p$. We let $X=C_k$ and claim that $X$ is as in the conclusion of the 
theorem. In order to see this, let $\vec{y}$ be such that $\fs(\vec{y})\subseteq C_k$.
 
Notice first that for distinct $n,m<\omega$ we must have 
$\tau(y_n)\neq\tau(y_m)$, for otherwise we would 
get, arguing in a similar way as in the proofs of Claims~\ref{atmosttwo} and~\ref{iszeroonpreviousrho}, 
that $\tau(y_n+y_m)=\tau(y_n)=\tau(y_m)$ and 
$\pi_{\tau(y_n+y_m)}(y_n+y_m)=\frac{2ki}{8}\neq\frac{ki}{8}$, a 
contradiction. Thus by rearranging $\vec{y}$ if necessary, we can assume 
that $n<m$ implies $\tau(y_n)<\tau(y_m)$.

Now an observation is in order. Consider 
$a\in[\omega]^{<\omega}\setminus\varnothing$ and 
$\varepsilon:a\longrightarrow\{1,2\}$. Let $m=\min(a)$ and 
$j=\tau(y_m)$. Since $\tau$ is increasing on $\vec{y}$, 
$\pi_j(y_n)=0$ for all $n\in a\setminus\{m\}$, 
while $\pi_j(y_m)=\frac{ki}{8}$. Thus 
\begin{equation*}
 \pi_j\left(\sum_{n\in a}\varepsilon(n)y_n\right)=\varepsilon(m)\frac{ki}{8}\neq0.
\end{equation*}

From this we can conclude that $\vec{y}$ satisfies 2-uniqueness of 
finite sums. Assume that $a,b\in[\omega]^{<\omega}$ and 
$\varepsilon:a\longrightarrow\{1,2\},\delta:b\longrightarrow\{1,2\}$ 
are such that 
\begin{equation}\label{exprescoef}
 \sum_{n\in a}\varepsilon(n)x_n=\sum_{n\in b}\delta(n)x_n.
\end{equation}
We will proceed by induction on $\min\{|a|,|b|\}$.
If $a=b=\varnothing$ we are done. Otherwise let $m=\min(a\cup b)$. 
Assume without loss of generality that $m\in a$, so that 
$m=\min(a)$. Let $j=\tau(y_m)$. Then by the previous observation, 
the value of each side of (\ref{exprescoef}) under $\pi_j$ is 
nonzero, while $\pi_j(y_n)=0$ for all $n>m$, thus by looking at 
the right-hand side of (\ref{exprescoef}) we conclude that 
we must have $m\in b$ as well. Then it is also the case that 
$\min(b)=m$. Now 
again, by the observation from last paragraph we get that 
the value of each side of (\ref{exprescoef}) under the 
function $\pi_j$ must equal, at the 
same time, $\varepsilon(m)\frac{ki}{8}$ and 
$\delta(m)\frac{ki}{8}$. This can only happen if 
$\varepsilon(m)=\delta(m)$, therefore we can cancel the 
term $\varepsilon(m)y_m$ from both sides of (\ref{exprescoef}) 
and get
\begin{equation*}
 \sum_{n\in a\setminus\{m\}}\varepsilon(n)x_n=\sum_{n\in b\setminus\{m\}}\delta(n)x_n,
\end{equation*}
now we can apply the inductive hypothesis and conclude that 
$a\setminus\{m\}=b\setminus\{m\}$ and 
$\varepsilon\upharpoonright(a\setminus\{m\})=\delta\upharpoonright(b\setminus\{m\})$. 
Since $m$ is an element of both $a$ and $b$, with 
$\varepsilon(m)=\delta(m)$, we have proved that $a=b$ and 
$\varepsilon=\delta$, and we are done.
\end{proof}

\begin{proof}[Proof of Theorem~\ref{mainresult}]
 Let $G$ be an abelian group, and $p\in G^*$ be a strongly summable 
 ultrafilter such that $\{x\in G\big|o(x)=2\}\notin p$. Since 
 $p$ is nonprincipal and the only $x\in G$ with $o(x)=1$ is $0$, 
 we have that $B=\{x\in G\big|o(x)>2\}\in p$. If 
 $C=\{x\in G\big|o(x)=3\}\in p$, then notice that, since  
 $C\subseteq\left\{x\in G\big|\pi_{\min(x)}(x)\neq\frac{1}{2}\right\}$ (because 
 $C=\left\{x\in G\big|(\forall n<\omega)\left(\pi_n(x)\in\left\{0,\frac{1}{3},-\frac{1}{3}\right\}\right)\right\}$), 
 we can apply Theorem~\ref{theoremjuris} and get an $X\in p$ such that, if 
 $\vec{x}$ is such that $\fs(\vec{x})\subseteq X$ (and there is such an $\vec{x}$ with 
 $\fs(\vec{x})\in p$ because of strong summability), then $\vec{x}$ must satisfy 2-uniqueness 
 of finite sums. If 
 $D=\{x\in G\big|o(x)=4\}\in p$, then we can pick a sequence 
 $\vec{x}$ such that $p\ni\fs(\vec{x})\subseteq D$, so by  
 Lemma~\ref{order4} this sequence must satisfy 2-uniqueness 
 of finite sums and we are done. Otherwise, if $C\notin p$ and $D\notin p$,  
 then 
 \begin{equation*}
  Q(G)=\{x\in G\big|o(x)>4\}=(G\setminus D)\cap(G\setminus C)\cap B\in p.
 \end{equation*}
 Now $Q(G)=Q_0\cup Q_1$, where 
 \begin{equation*}
  Q_0=\left\{x\in Q(G)\bigg|\pi_{\rho(x)}(x)\notin\left\{\frac{1}{8},-\frac{1}{8},\frac{3}{8},-\frac{3}{8}\right\}\right\},
 \end{equation*}
 and
 \begin{equation*}
  Q_1=\left\{x\in Q(G)\bigg|\pi_{\rho(x)}(x)\in\left\{\frac{1}{8},-\frac{1}{8},\frac{3}{8},-\frac{3}{8}\right\}\right\},
 \end{equation*}
 so pick $i\in 2$  such that $Q_i\in p$. If $i=0$ apply 
 Lemma~\ref{eluncuarto} and if $i=1$ apply Theorem~\ref{micontri}, 
 in either case, there is an $X\in p$ such that whenever $\vec{x}$ 
 is such that $\fs(\vec{x})\subseteq X$, then $\vec{x}$ must 
 satisfy 2-uniqueness of finite sums. By strong summability of 
 $p$ there is such a sequence $\vec{x}$ which additionally 
 satisfies $\fs(\vec{x})\in p$, and we are done.
\end{proof}


\begin{cor}[\cite{jurisetal}, Question 4.12]
 Let $p$ be a nonprincipal strongly summable ultrafilter on an abelian 
 group $G$. Then $p$ is sparse.
\end{cor}

\begin{proof}
 Let $G$ be any abelian group, and let $p\in G^*$ be a 
 strongly summable ultrafilter. Let 
 \begin{equation*}
  B=\{x\in G\big|o(x)\leq2\}.
 \end{equation*}
 Then $B$ is a subgroup of $G$. If $B\in p$ then 
 since $p$ is nonprincipal, $B$ must be infinite; 
 and since $G$ is countable, $B$ must be isomorphic to the 
 (unique up to isomorphism) countably infinite 
 Boolean group. Consider the restricted 
 ultrafilter $q=p\upharpoonright B=p\cap\mathfrak P(B)$. 
 Then $q$ is also strongly summable, so $q$ is 
 a nonprincipal strongly summable ultrafilter 
 on the Boolean group and therefore by 
 \cite[Th. 2.1]{yonilaboriel} it is sparse. It 
 is easy to see that this implies that $p$ is 
 sparse as well. Thus the only case that remains 
 to be proved is when $B\notin p$, but this is 
 handled by Theorem~\ref{mainresult} together with 
 Corollary~\ref{eldelesparcido}, and we are done.
\end{proof}

\begin{cor}[\cite{jurisetal}, Question 4.11]
 Let $p$ be a nonprincipal strongly summable ultrafilter on an abelian group $G$. Then $p$ has the trivial 
 sums property.
\end{cor}

\begin{proof}
 Let $G$ be any abelian group, and let $p\in G^*$ be a 
 strongly summable ultrafilter. If $p$ does not contain 
 the subgroup $B=\{x\in G\big|o(x)\leq2\}$, then 
 we just need to apply Theorems~\ref{mainresult} 
 and~\ref{eldelassumas}. So assume that $B\in p$ 
 and let $q,r\in\beta G$ be such that $q+r=p$. 
 Then we have that 
 \begin{equation*}
  \{x\in G\big|B-x\in r\}\in q,
 \end{equation*}
 in particular this set is nonempty and so we 
 can pick an $x\in G$ such that $B-x\in r$, or 
 equivalently $B\in r+x$. Since 
 $x\in G$ (hence it commutes with all ultrafilters), the equation $(q-x)+(r+x)=p$ holds, 
 thus 
 \begin{equation*}
  A=\{y\in G\big|B-y\in r+x\}\in q-x.
 \end{equation*}
 Notice that $A\subseteq B$, because if $y\in G$ 
 is such that $B-y\in r+x$ then 
 $B\cap(B-y)\in r+x$, in particular the latter 
 set is nonempty and so there are $z,w\in B$ 
 such that $z=w-y$ which means that $y=w-z\in B$. 
 Therefore $B\in q-x$, so we can define 
 $u=(q-x)\upharpoonright B$ and 
 $v=(r+x)\upharpoonright B$. We then get that 
 $u,v\in\beta B$ and $p\upharpoonright B\in B^*$ 
 is a strongly summable ultrafilter such that 
 $u+v=p\upharpoonright B$. Notice that in $B$, 
 FS-sets are just subgroups from which the 
 element $0$ might have been removed; thus the 
 filter $\{A\cup\{0\}\big|A\in p\upharpoonright B\}$ 
 has a base of subgroups and hence it is the neighbourhood 
 filter of $0$ for some group topology. This means 
 that $p\upharpoonright B$ satisfies the hypothesis 
 of \cite[Cor. 4.4]{protasums}, so it must be the 
 case that 
 $u,v\in B+p\upharpoonright B$. 
 This is easily seen to 
 imply that $q-x,r+x\in B+p$, and therefore, 
 since $x\in G$, we conclude that $q,r\in G+p$ and we 
 are done.
\end{proof}

\section{The Boolean group}

Theorem~\ref{mainresult} from the previous section depends heavily 
on the hypothesis that the ultrafilter $p$ at hand does not contain the subgroup 
$B(G)=\{x\in G\big|o(x)=2\}$, since 
there are no sequences $\vec{x}$ satisfying the 2-uniqueness of finite sums in 
$B(G)$. Corollary~\ref{nonbooleanimpliesunion} also has that $B(G)\notin p$ as a 
hypothesis, but it is not entirely clear \textit{a priori} that this hypothesis 
is necessary for the result. The main objective of this section is to prove that we 
do in fact need such a hypothesis. This is, if $p\in G^*$ is strongly summable and 
$B(G)\in p$, then there is no guarantee that $p$ is additively isomorphic to a 
union ultrafilter. For this, of course, we only need to consider the case where 
$B(G)$ is infinite (otherwise, the only ultrafilters that can contain it are 
the principal ones). And, as noted in the previous section, when dealing with 
strongly summable ultrafilters we may assume without loss of generality that 
$G$ (and hence $B(G)$) is countable. Since there is (up to isomorphism) only one 
countably infinite 
group all of whose nonidentity elements have order $2$, it will be enough for our 
purposes to look at strongly summable ultrafilters on this group (which we 
will from now on simply call ``the Boolean group''), by focusing our attention on 
the restricted ultrafilter $p\upharpoonright B(G)$.

We will choose a particularly nice ``realization'' of the Boolean group to 
work with. We think of the Boolean group as the set 
$\mathbb B=[\omega]^{<\omega}$ equipped with the 
symmetric difference $\bigtriangleup$ as group operation. Since every element 
of $\mathbb B$ has order 2, we have that for any sequence $\vec{x}$ of elements of 
$\mathbb B$, we can ignore the repeated elements from the sequence and still get 
the same set $\fs(\vec{x})$. Thus we will talk about $\fs(X)$ for $X\subseteq\mathbb B$, 
and it is easy to see that for $p\in\mathbb B^*$, $p$ is strongly summable if and only 
if for every 
$A\in p$ there is an infinite set $X\subseteq\mathbb B$ such that 
$p\ni\fs(X)\subseteq A$.

We will use the fact that $\mathbb B$ is a vector space over the field with two 
elements 
$\mathbb F_2=\mathbb Z/2\mathbb Z$ (scalar multiplication being the obvious one). 
Note that for $X\subseteq\mathbb B$, the subspace spanned (which in $\mathbb B$ 
coincides with the subgroup generated) by $X$ is 
exactly $\fs(X)\cup\{\varnothing\}$, because nontrivial linear combinations 
(i.e. linear combinations in which not all scalars equal zero) of 
elements of $X$ are exactly finite sums (or symmetric differences) of 
elements of $X$. The following proposition, whose proof is obvious, tells us 
how do 
subsets $X\subseteq G$ satisfying uniqueness of finite sums look like.

\begin{prop}\label{uniquenessli}
 For $X\subseteq G$, the following are equivalent:
 \begin{enumerate}[(i)]
  \item\label{uniqueness} $X$ satisfies uniqueness of finite sums.
  \item\label{0notin} $\varnothing\notin\fs(X)$.
  \item\label{linind} $X$ is linearly independent.
 \end{enumerate}
\end{prop}

Thus when we have a set $\fs(Y)$ such that $Y$ is not linearly independent, we can always choose a 
basis $X$ for the subspace $\fs(Y)$ spanned by $Y$, and we will have that 
$\fs(X)=\fs(Y)\setminus\{\varnothing\}$. This means that, when considering sets of the form 
$\fs(X)$, we can assume without loss of generality that $X$ is linearly independent. Another 
way to see this is the following: let
$p\in B^*$ be a strongly summable ultrafilter, 
and let $A\in p$. Since $p$ is nonprincipal, $\{\varnothing\}\notin p$ and hence 
$A\setminus\{\varnothing\}\in p$. Therefore we can choose an $X$ such that 
$p\ni\fs(X)\subseteq A\setminus\{0\}$, so $\fs(X)\subseteq A$ and $X$ must 
be linearly independent.

\begin{definition}
 For a linearly independent set $X\subseteq\mathbb B$, we define for an element $y\in\fs(X)$ the 
 \textbf{$X$-support} of $y$, denoted by $X-\supp(y)$, as the (unique, by linear independence of $X$) 
 finite set of elements of $X$ whose 
 sum equals $y$. This is,
 \begin{equation*}
  y=\sum_{x\in X-\supp(y)}x.
 \end{equation*}
 If $Y\subseteq\fs(X)$ then, we also define the $X$-support of $Y$ as 
 \begin{equation*}
  X-\supp(Y)=\union_{y\in Y}X-\supp(y).
 \end{equation*}
 Similarly, we define the $X$-support of a sequence of elements of $\fs(X)$ as the $X$-support 
 of its range.
\end{definition}

It will be convenient to stipulate the convention that $X-\supp(\varnothing)=\varnothing$. 
Then it is 
readily checked that the function 
$X-\supp:\fs(X)\cup\{\varnothing\}\longrightarrow([X]^{<\omega},\bigtriangleup)$ is a group 
isomorphism (in fact, a linear transformation between the two vector spaces),
in other words, $X-\supp(x\bigtriangleup y)=X-\supp(x)\bigtriangleup X-\supp(y)$ for 
all $x,y\in\fs(X)$; and more generally 
$X-\supp\left(\sum_{x\in A}x\right)=\sum_{x\in A}X-\supp(x)$ for all $A\in[\fs(X)]^{<\omega}$. 
This is the really crucial feature of the $X$-support, and it will be used 
ubiquitously in what follows.

As an application of the previous definitions and properties, we will provide another proof 
of the fact that every strongly 
summable ultrafilter on $\mathbb B$ is sparse, much simpler than the original one from 
\cite[Th. 2.1]{yonilaboriel}. So let $p\in\mathbb B^*$ be a strongly summable ultrafilter, and 
let $A\in p$. Because of strong summability, there is an infinite linearly 
independent $Z$ such that $p\ni\fs(Z)\subseteq A$.

\begin{claim}\label{existsinfty}
 There is a $B\in p$ such that for some infinite $W\subseteq Z$, $\fs(W)\cap B=\varnothing$.
\end{claim}

The result follows easily from the claim: just pick a linearly independent $Y$ such that 
$p\ni\fs(Y)\subseteq B\cap\fs(Z)$, and let $X=Y\cup W$. Then it is straightforward to prove that 
$X$ is linearly independent, since so are $Y$ and $W$, and $\fs(W)$ is disjoint from $\fs(Y)$. 
Since $X\setminus Y=W$ we also have that $|X\setminus Y|=\omega$; and since 
$Y,W\subseteq\fs(Z)$, we 
will have that $\fs(X)\subseteq\fs(Z)\subseteq A$ and we are done.

\begin{proof}[Proof of Claim~\ref{existsinfty}]
 Let $Z'$ be an infinite, coinfinite subset of $Z$. Let
 \begin{equation*}
  B_0=\left\{w\in\fs(Z)\middle|Z-\supp(w)\cap Z'\neq\varnothing\right\},
 \end{equation*}
 \begin{equation*}
  B_1=\fs(Z)\setminus B_0=\left\{w\in\fs(Z)\middle|Z-\supp(w)\cap Z'=\varnothing\right\}.
 \end{equation*}

 There is $i\in2$ such that $B_i\in p$. If $B_0\in p$ then we let $W=Z\setminus Z'$; otherwise if 
 $B_1\in p$ we let $W=Z'$. In any case it is easy to see that $\fs(W)\cap B_i=\varnothing$.
\end{proof}

The rest of this section is devoted to showing that the hypothesis that 
$\{x\in G\big|o(x)=2\}\notin p$ in Corollary~\ref{nonbooleanimpliesunion} is necessary, 
by constructing a nonprincipal strongly summable 
ultrafilter on $\mathbb B$ that is not additively isomorphic to a union ultrafilter. This 
construction borrows lots of ideas from the constructions of unordered union 
ultrafilters that can be found in \cite[Th. 4]{blasshindman} and 
\cite[Cor. 5.2]{krautz-union}. We first show an effective way to look at additive isomorphisms to 
union ultrafilters.

\begin{lem}\label{caracaddisom}
 Let $p\in\mathbb B^*$ be a strongly summable ultrafilter that is additively isomorphic to some 
 union ultrafilter. Then there exists a linearly independent $X$ such that $\fs(X)\in p$ and satisfying  
 that whenever $A\subseteq\fs(X)$ is such that $A\in p$, there exists a set $Z$, whose elements have 
 pairwise disjoint $X$-supports, with $p\ni\fs(Z)\subseteq A$.
\end{lem}

\begin{proof}
 If the strongly summable ultrafilter $p\in\mathbb B^*$ is additively 
 isomorphic to a union ultrafilter, by Propositions~\ref{addisomissimple} 
 and~\ref{uniquenessli}, we 
 have that for 
 some linearly independent $X$ such that $\fs(X)\in p$ and for some enumeration of $X$ as 
 $X=\{x_n\big|n<\omega\}$, the mapping $\varphi:\fs(X)\longrightarrow[\omega]^{<\omega}$ 
 given by 
 $\sum_{n\in a}x_n\longmapsto a$ sends $p$ to a union ultrafilter. Note that the mapping 
 $\varphi$ is 
 a vector space isomorphism from the subspace spanned by $X$, to all of $\mathbb B$ (in fact 
 it is 
 the unique linear extension of the mapping $x_n\longmapsto\{n\}$). The fact that 
 $\varphi(p)$ is a union 
 ultrafilter 
 means that, for every $A\subseteq\fs(X)$ such that $A\in p$, there is a pairwise disjoint family $Y$ 
 such that $\varphi(p)\ni\fu(Y)\subseteq\varphi[A]$. Since $Y$ is pairwise disjoint, we 
 get that $\fu(Y)=\fs(Y)$ and since $\varphi$ is an isomorphism, $\varphi^{-1}[\fs(Y)]=\fs(Z)$ 
 where $Z=\varphi^{-1}[Y]$. Now the fact that $Y$ is 
 pairwise disjoint means that the $X$-supports of the elements of $Z$ are pairwise disjoint, and 
 we have that $p\ni\fs(Z)\subseteq A$.
\end{proof}

Thus our goal is to construct, by a transfinite recursion, a strongly summable ultrafilter and 
somehow, at the same time, for each linearly independent $X$ such that $\fs(X)$ will end up 
in the ultrafilter, at some stage we need to start making sure that, for every new set of 
the form $\fs(Z)$ that we are adding to the ultrafilter, the generators $Z$ do not have 
pairwise disjoint $X$-support. The notions of suitable and adequate families for $X$ will 
precisely code the way in which we are going to ensure that.

\begin{definition}\label{defsuitability}
 For a linearly independent subset $X\subseteq G$, we will say that a subset $Y\subseteq\fs(X)$ 
 is \textbf{suitable} for $X$ if:
 \begin{enumerate}[(i)]
  \item\label{suitawitness} For each $m<\omega$ there exists an $m$-sequence 
  $\langle y_i\big|i<m\rangle$ of elements of 
  $Y$ such that whenever $i<j<m$, the set $X-\supp(y_i)\cap X-\supp(y_j)$ is nonempty. This 
  sequence will be called an \textbf{$m$-witness for suitability}.
  \item\label{elcachito} Whenever $y,y'\in Y$ are such that $X-\supp(y)\cap X-\supp(y')$ 
  is nonempty, the set
  $[X-\supp(y)\cap X-\supp(y')]\setminus X-\supp(Y\setminus\{y,y'\})$ is also nonempty. 
  (We do not require here that $y\neq y'$; in particular, for each $y\in Y$, 
  $X-\supp(y)\setminus X-\supp(Y\setminus\{y\})$ is nonempty, and 
 this is easily seen to imply that $Y$ must be linearly independent).
 \end{enumerate}
\end{definition}

Thus a suitable set $Y$ for $X$ contains, in a carefully controlled way, arbitrarily large 
bunches of elements whose $X$-supports always pairwise intersect. Given a linearly 
independent set 
$X$, it is easy to inductively build a set $Y$ that is suitable for $X$. And once we have 
such a suitable set, we can look at subsets of $\fs(Y)$ which, in a sense, borrow 
from $Y$ the non-disjointness of their $X$-supports. This is captured in a precise sense 
by the following definition, which also captures the fact that we will want to handle the 
non-disjointness of the $X$-supports for several distinct 
linearly independent sets $X$ simultaneously.

\begin{definition}\label{defadequacy}
 Let $A\subseteq\mathbb B$ and let $\mathscr Y=\{(X_i,Y_i)\big|i<n\}$ be a finite family 
 such that for each $i<n$, $X_i$ is a linearly independent subset of $G$ and $Y_i$ is 
 suitable for $X_i$. Also, let $m<\omega$. Then we will say that $A$ is 
 \textbf{$(\mathscr Y,m)$-adequate} if there exists an $m$-sequence 
 $\langle a_j\big|j<m\rangle$, called a 
 \textbf{$(\mathscr Y,m)$-witness for adequacy}, such that for each $i<n$, 
 \begin{enumerate}[(i)]
  \item\label{adequafsin} $\fs(\vec{a})\subseteq A\cap\fs(Y_i)$ (which is in turn a subset 
  of $\fs(X_i)$),
  \item\label{adequawitness} There exists an $m$-witness for the suitability of $Y_i$, 
  $\langle y_j\big|j<m\rangle$, such that for each two distinct $j,k<m$, $y_j\in Y_i-\supp(a_j)$ and 
  $y_j\notin Y_i-\supp(a_k)$.
 \end{enumerate}
 If we are given a family of ordered pairs $\mathscr X$ all of whose first entries 
 are linearly independent subsets of $\mathbb B$, while every second entry is 
 suitable for the corresponding first entry, then we will say that $A$ is 
 $\mathscr X$-adequate if it is $(\mathscr Y,m)$-adequate for all finite 
 $\mathscr Y\subseteq\mathscr X$ and for all $m<\omega$. When $\mathscr Y$ is a 
 singleton $\{(X,Y)\}$, we will just say that $A$ is $(X,Y)$-adequate.
\end{definition}

Requirement~(\ref{adequawitness}) of Definition~\ref{defadequacy} in particular 
implies that, for $j<k<m$, the set $X_i-\supp(a_j)\cap X_i-\supp(a_k)$ is 
nonempty. Thus the $X_i$-supports of the terms of a witness for adequacy are 
not pairwise disjoint, and moreover their non-disjointness does not happen 
randomly, but is rather induced by some non-disjointness going on at the level 
of $Y_i$. Also, note that if $Y$ is suitable for $X$ then $\fs(Y)$ is 
$(X,Y)$-adequate, with the witnesses for suitability witnessing adequacy at 
the same time. The following lemma, along with the observation that an 
$\mathscr X$-adequate set is also $(X,Y)$-adequate for each 
$(X,Y)\in\mathscr X$, tells us that this notion of adequacy is 
adequate (pun intended) for our purpose of banishing sets of the form 
$\fs(Z)$ for which the elements of $Z$ have pairwise disjoint 
$X$-supports.

\begin{lem}\label{matedisjuntez}
Let $X$ and $Z$ be both linearly independent and let $Y$ be suitable for $X$. 
Assume that $Z\subseteq\fs(Y)$. If the elements of $Z$ have pairwise disjoint 
$X$-supports then $\fs(Z)$ is not $(X,Y)$-adequate.
\end{lem}

\begin{proof}
Clause~(\ref{elcachito}) from Definition~\ref{defsuitability} implies that, 
for two distinct $z,z'\in Z$, if $y\in Y-\supp(z)$ and $y'\in Y-\supp(z')$ 
then $X-\supp(y)\cap X-\supp(y')=\varnothing$, for otherwise 
$X-\supp(z)$ would not be disjoint from $X-\supp(z')$. Thus 
$\langle z,z'\rangle$ cannot be an $((X,Y),2)$-witness. More generally, 
for any two $w,w'\in\fs(Z)$, the only way that there could exist two 
distinct $y\in Y-\supp(w)$ and $y'\in Y-\supp(w')$ such that 
$X-\supp(y)\cap X-\supp(y')\neq\varnothing$ would be if 
$y,y'\in Y-\supp(z)$ for some $z\in Z$ such that 
$z\in Z-\supp(w)\cap Z-\supp(w')$. But then 
$y\in Y-\supp(w')$ and $y'\in Y-\supp(w)$. Hence 
$\langle w,w'\rangle$ cannot be an $((X,Y),2)$-witness and we are done.
\end{proof}

Given this, the idea for the recursive construction of an ultrafilter 
would be as follows: at each stage we choose some set $\fs(X)$ 
that has already been added to the ultrafilter, and then we 
choose a suitable (for $X$) set $Y$. At every stage we make sure that the 
subsets of $\mathbb B$ that we are adding to the ultrafilter 
are $\mathscr X$-adequate, where $\mathscr X$ is the collection 
of all pairs $(X,Y)$ that have been thus chosen so far. If we want 
to have a hope of succeeding in such a construction, we better 
make sure that the notion of being $\mathscr X$-adequate 
behaves well with respect to partitions. For this we will 
need the following lemma.

\begin{lem}\label{condensatestigos}
Let $\mathscr Y=\{(X_i,Y_i)\big|i<n\}$ where each $X_i$ is 
linearly independent and each $Y_i$ is suitable for $X_i$. 
Let $\vec{a}=\langle a_j\big|j<M\rangle$ be a 
$(\mathscr Y,M)$-witness for adequacy, and let 
$\langle b_i\big|i<m\rangle$ be an $m$-sequence of pairwise 
disjoint subsets of $M$. If we define $\vec{c}=\langle c_j\big|j<m\rangle$ by 
$c_j=\sum_{k\in b_j}a_k$, then $\vec{c}$ will be a 
$(\mathscr Y,m)$-witness for adequacy.
\end{lem}

\begin{proof}
 Let us check that $\vec{c}$ satisfies both requirements of 
 Definition~\ref{defadequacy} for a 
 $(\mathscr Y,m)$-witness. Fix $i<n$. Since the $b_j$ are pairwise 
 disjoint, we have that 
 $\fs(\vec{c})\subseteq\fs(\vec{a})\subseteq A\cap\fs(Y_i)$, thus 
 requirement~(\ref{adequafsin}) is satisfied. In order to see that 
 requirement~(\ref{adequawitness}) holds, grab the corresponding $m$-witness for suitability, 
 $\langle y_j\big|j<M\rangle$, as in part~(\ref{adequawitness}) of Definition~\ref{defadequacy} 
 for $\vec{a}$. Now for $j<m$, pick a 
 $k_j\in b_j$ and let $w_j=y_{k_j}$. Since the $w_j$ were chosen 
 from among the $y_k$, the sequence $\vec{w}=\langle w_j\big|j<m\rangle$ is an $m$-witness 
 for suitability. 
 Now for $j<m$, since $w_j\in Y_i-\supp(a_{k_j})$ and 
 $w_j\notin Y_i-\supp(a_l)$ for $l\neq k_j$, it follows that 
 $w_j\in Y_i-\supp(c_j)$ and 
 $w_j\notin Y_i-\supp(c_{j'})$ for $j\neq j'$, and we are done.
\end{proof}

An easy consequence of the previous lemma is the observation 
that any $(\mathscr Y,M)$-adequate set is also 
$(\mathscr Y,m)$-adequate for any $m\leq M$. 
Lemma~\ref{condensatestigos} will allow us to prove the 
following lemma, which is crucial.

\begin{lem}\label{partitadequatefinitary}
For each $m<\omega$ there is an $M<\omega$ such that whenever $\mathscr Y$ is a finite 
family of 
ordered pairs of the 
form $(X,Y)$, with $X$ a linearly independent set and $Y$ suitable for $X$, and whenever a 
$(\mathscr Y,M)$-adequate set is partitioned into two cells, one of 
the cells must be $(\mathscr Y,m)$-adequate.
\end{lem}

\begin{proof}
For this, we will use a theorem of Graham and Rothschild 
which is a finitary 
version of Hindman's theorem, namely: for every $m<\omega$ 
there is an $M<\omega$ such that whenever we partition 
$\mathfrak P(M)\setminus\{\varnothing\}$ into two cells, 
then one of the cells contains $\fu(\vec{b})$ for 
some pairwise disjoint $m$-sequence 
$\vec{b}=\langle b_i\big|i<m\rangle$ of nonempty subsets of $M$ (this 
result is sometimes referred to as the \emph{Folkman-Rado-Saunders} theorem).
An elegant proof of this theorem from the infinitary version, using 
a so-called compactness argument, 
can be obtained by following the proof of 
\cite[Th. 5.29]{hindmanstrauss} as a template, applied to 
the semigroup whose underlying set is $[\omega]^{<\omega}$ and whose 
semigroup operation is the union $\cup$.

Thus for $m<\omega$, let $M$ be given by this 
finitary theorem, and let $A$ be a 
$(\mathscr Y,M)$-adequate set. Let 
$\vec{a}=\langle a_j\big|j<M\rangle$ be 
a $(\mathscr Y,M)$-witness 
for the adequacy of $A$. If $A$ is partitioned into the two 
cells $A_0,A_1$, then since $\fs(a)\subseteq A$, 
we can induce a partition of 
$\mathfrak P(M)\setminus\{\varnothing\}$ into the two 
cells $B_0,B_1$ by declaring a subset $s\subseteq M$ 
to be an element of $B_l$ iff $\sum_{j\in s}a_j\in A_l$ 
for $l\in2$. Then the theorem of Graham and Rothschild gives 
us a pairwise disjoint family $\vec{b}=\langle b_j\big|j<m\rangle$ 
and an $l\in2$ such that $\fu(\vec{b})\subseteq B_l$. Letting 
$\vec{c}=\langle c_j\big|j<m\rangle$ be given by 
$c_j=\sum_{k\in b_j}a_k$, we get that $\fs(\vec{c})\subseteq A_l$ 
and Lemma~\ref{condensatestigos} ensures that $\vec{c}$ is 
a $(\mathscr Y,m)$-witness for adequacy. Therefore $A_l$ is 
$(\mathscr Y,m)$-adequate and we are done.
\end{proof}

\begin{cor}\label{partitadequateinfinitary}
For any family $\mathscr X$ consisting of ordered pairs of the 
form $(X,Y)$, with $X$ a linearly independent set and $Y$ suitable 
for $X$, if we partition an $\mathscr X$-adequate 
set into two cells, then one of them must be 
$\mathscr X$-adequate.
\end{cor}

\begin{proof}
If $A=A_0\cup A_1$ is a partition of the $\mathscr X$-adequate set $A$, 
and neither $A_0$ nor $A_1$ are $\mathscr X$-adequate, then the reason for this is 
the existence of finite 
$\mathscr Y_0,\mathscr Y_1\subseteq\mathscr X$ and 
$m_0,m_1<\omega$ such that $A_0$ is not 
$(\mathscr Y_0,m_0)$-adequate and $A_1$ is not 
$(\mathscr Y_1,m_1)$-adequate. Pick the $M$ that 
works for $\max\{m_0,m_1\}$ in 
Lemma~\ref{partitadequatefinitary}. Then for some 
$i\in2$, $A_i$ is 
$(\mathscr Y_0\cup\mathscr Y_1,\max\{m_0,m_1\})$-adequate 
(because $A$ is $(\mathscr Y_0\cup\mathscr Y_1,M)$-adequate), 
in particular $A_i$ is $(\mathscr Y_i,m_i)$-adequate, 
a contradiction.
\end{proof}

Recall that, in an abstract setting, if we have a 
set $X$ and a family $\mathscr A\subseteq\mathfrak P(X)$ 
then we say that $\mathscr A$ is \textbf{partition regular}, 
or a \textbf{coideal}, 
if $\mathscr A$ is closed under supersets and, whenever 
an element of $\mathscr A$ is partitioned into two cells, 
the family $\mathscr A$ necessarily contains at least one of 
the cells. Thus the previous corollary establishes that, 
for any family $\mathscr X$, the collection of 
$\mathscr X$-adequate subsets of $\mathbb B$ is partition 
regular. This is important because of the well-known 
fact that, if $\mathscr A$ is partition regular and 
$\mathcal F\subseteq\mathscr A$ is a filter on $X$, 
then it is possible to extend $\mathcal F$ 
to an ultrafilter $p\subseteq\mathscr A$.

With these preliminary results under our belt, we are finally ready to 
prove the main theorem of this section.

\begin{thrm}\label{crazyultrafilter}
 If $\covm=\mathfrak c$, then there exists a strongly summable 
 ultrafilter on $\mathbb B$ that is not additively isomorphic 
 to any union ultrafilter.
\end{thrm}

\begin{proof}
 Let $\{A_\alpha\big|\alpha<\mathfrak c\}$ be an enumeration of 
 all subsets of $\mathbb B$, and let 
 $\langle X_\alpha\big|\alpha<\mathfrak c\rangle$ be an 
 enumeration of all infinite linearly independent subsets of $\mathbb B$ 
 in such a way that each such set appears 
 cofinally often in the enumeration. Now 
 recursively define linearly independent sets 
 $\langle Y_\alpha\big|\alpha<\mathfrak c\rangle$ and a 
 strictly increasing sequence of ordinals 
 $\langle \gamma_\alpha\big|\alpha<\mathfrak c\rangle$ 
 satisfying the following conditions for each $\alpha<\mathfrak c$:

 \begin{enumerate}[(i)]
  \item\label{whatisgama} $\gamma_\alpha$ is the least 
  $\eta\geq\sup_{\xi<\alpha}(\gamma_\xi+1)$ such that 
  $\fs(Y_\xi)\subseteq\fs(X_\eta)$ for some 
  $\xi<\alpha$.
  \item\label{suitability} $Y_\alpha$ is suitable for $X_{\gamma_\alpha}$.
  \item\label{ultra} $\fs(Y_\alpha)$ is either contained in or disjoint 
  from $A_\alpha$.
  \item\label{filterness} The family 
  $\mathcal F_\alpha=\{\fs(Y_\xi)\big|\xi\leq\alpha\}$ is centred.
  \item\label{adequacy} Letting $\mathscr X_\alpha=\{(X_{\gamma_\xi},Y_\xi)\big|\xi\leq\alpha\}$, 
  the filter generated by $\mathcal F_\alpha$ consists of $\mathscr X_\alpha$-adequate sets.
 \end{enumerate}

 Thus at each stage $\alpha$, we first use clause~(\ref{whatisgama}) to 
 determine what $\gamma_\alpha$ will be, and then we work to find a 
 $Y_\alpha$ satisfying (\ref{suitability})--(\ref{adequacy}).

 Let us first look at what we have at the end of this construction. 
 Clause~(\ref{filterness}) tells us that the family  
 $\{\fs(Y_\alpha)\big|\alpha<\mathfrak c\}$ generates a filter $p$, 
 which will be an ultrafilter because of (\ref{ultra}), and it will 
 obviously be nonprincipal and strongly summable. Now notice 
 that (\ref{adequacy}) implies that, if 
 $\mathscr X_{\mathfrak c}=\{(X_{\gamma_\alpha},Y_\alpha)\big|\alpha<\mathfrak c\}$, 
 then each $A\in p$ will be $\mathscr X_{\mathfrak c}$-adequate, because if 
 $\mathscr Y=\{(X_{\gamma_{\alpha_i}},Y_i)\big|i<n\}$ is a finite 
 subfamily of $\mathscr X_{\mathfrak c}$, $m<\omega$, and $A\in p$, then we can 
 grab an $\alpha<\mathfrak c$ larger than all $\gamma_{\alpha_i}$ 
 and also larger than the $\beta$ witnessing 
 $\fs(Y_\beta)\subseteq A$. By (\ref{adequacy}), 
 $\fs(Y_\alpha)\cap\fs(Y_\beta)$ is $\mathscr X_\alpha$-adequate, in particular 
 it is $(\mathscr Y,m)$-adequate and thus so is $A$.

 The last observation is crucial for the argument that $p$ cannot 
 be additively isomorphic to any union ultrafilter. If it was, by 
 Lemma~\ref{caracaddisom} there would be a linearly 
 independent $X$ such that $\fs(X)\in p$ and such that for each 
 $A\in p$ satisfying $A\subseteq\fs(X)$, we would be able to 
 find a family $Z$ whose elements have pairwise disjoint 
 $X$-supports and such that $p\ni\fs(Z)\subseteq A$. Now since 
 $\fs(X)\in p$, there is an $\alpha<\mathfrak c$ such that 
 $\fs(Y_\alpha)\subseteq\fs(X)$, let $\eta$ be the least 
 ordinal $\geq\sup_{\xi\leq\alpha}(\gamma_\xi+1)$ such that $X=X_\eta$. By~(\ref{whatisgama}) we 
 will have that $\gamma_{\alpha+1}\leq\eta$ and, in fact, whenever $\xi>\alpha$ is
 such that no $\gamma_\beta$ equals $\eta$ for any 
 $\alpha<\beta<\xi$, then $\gamma_\xi\leq\eta$. Thus there will eventually be some 
 $\zeta>\alpha$ such that $\gamma_\zeta=\eta$, and 
 by~(\ref{suitability}) this 
 means that $Y_\zeta$ is suitable for $X$. Since every element 
 of $p$ is $\mathscr X_{\mathfrak c}$-adequate, in particular 
 $(X,Y_\zeta)$-adequate, then by Lemma~\ref{matedisjuntez} we get 
 that for no set $Z$ with pairwise disjoint $X$-supports can 
 we have that $p\ni\fs(Z)\subseteq\fs(Y_\zeta)$. This shows 
 that $p$ cannot be additively isomorphic to any union ultrafilter, 
 and we are done.

 We now proceed to show how is it possible to carry out such 
 a construction. So let $\alpha<\mathfrak c$ and assume 
 that for all $\xi<\alpha$, conditions 
 (\ref{whatisgama})--(\ref{adequacy}) are satisfied. As 
 mentioned before, condition~(\ref{whatisgama}) 
 uniquely determines $\gamma_\alpha$, so we only need 
 to focus on constructing $Y_\alpha$ satisfying 
 conditions~(\ref{suitability})--(\ref{adequacy}).
 Let $\mathcal F=\{\fs(Y_\xi)\big|\xi<\alpha\}$, 
 and $\mathscr X=\{(X_{\gamma_\xi},Y_\xi)\big|\xi<\alpha\}$.
 Condition~(\ref{adequacy}) implies that the filter 
 generated by $\mathcal F$ consists of $\mathscr X$-adequate 
 sets, if $\alpha$ is limit, by the same argument 
 as in the proof that $p$ consists of  
 $\mathscr X_{\mathfrak c}$-adequate sets, and if $\alpha=\xi+1$ 
 just because $\mathcal F=\mathcal F_\xi$ and 
 $\mathscr X=\mathscr X_\xi$. Thus if we define 
 \begin{equation*}
  H=\left\{q\in\beta\mathbb B\bigg|(q\supseteq\mathcal F)\wedge(\forall A\in q)(A\mathrm{\ is\ }\mathscr X\mathrm{-adequate})\right\},
 \end{equation*}
 then $H$ will be a nonempty subset of $\beta\mathbb B$ by 
 Corollary~\ref{partitadequateinfinitary} (cf. the discussion 
 following that Corollary). Since finite sets cannot be $\mathscr{X}$-adequate, 
 we have that, in fact, $H\subseteq B^*$. In what follows, in order to 
 avoid confusion, we will use the symbol $\blacktriangle$ to 
 denote the extension of the group operation $\bigtriangleup$ 
 on $\mathbb B$ to all of $\beta\mathbb B$. We will also 
 use that symbol to denote translates of sets, 
 $x\blacktriangle A=\{x\bigtriangleup y\big|y\in A\}$. Thus, 
 with this notation, 
 \begin{equation*}
  p\blacktriangle q=\{A\subseteq\mathbb B\big|\{x\in\mathbb B\big|x\blacktriangle A\in q\}\in p\}.
 \end{equation*}

 \begin{claim}\label{subsemigroup}
 H is a closed subsemigroup of $\mathbb B$.
 \end{claim}

 \begin{proof}[Proof of Claim]
  The fact that $H$ is closed is fairly straightforward and 
  is left to the 
  reader. To prove that $H$ is a subsemigroup, let $p,q\in H$. 
  We first show that $\mathcal F\subseteq p\blacktriangle q$. 
  Fix a $\xi<\alpha$, and note that we have, for each $w\in\fs(Y_\xi)$, 
  that $w\blacktriangle\fs(Y_\xi)=\fs(Y_\xi)\cup\{\varnothing\}\in q$. 
  Hence $p\ni\fs(Y_\xi)\subseteq\{x\in\mathbb B\big|x\blacktriangle\fs(Y_\xi)\in q\}$, 
  which means that $\fs(Y_\xi)\in p\blacktriangle q$.
  
  Now we 
  only need to show that, if $A\in p\blacktriangle q$, then 
  $A$ is $\mathscr X$-adequate. So fix a finite  
  $\mathscr Y=\{(X_i,Y_i)\big|i<n\}\subseteq\mathscr X$ and an 
  $m<\omega$. We will 
  see that there is a $(\mathscr Y,m)$-witness for the
  adequacy of $A$. Let 
  $B=\{x\in\mathbb B\big|x\blacktriangle A\in q\}$. 
  We have that $B\in p$ because $A\in p\blacktriangle q$, so $B$ is 
  $\mathscr X$-adequate and thus we can grab a $(\mathscr Y,m)$-witness 
  $\langle a_j\big|j<m\rangle$ for the adequacy of $B$. For each $i<n$,
  $\fs(\vec{a})\subseteq\fs(Y_i)$ so we can define 
  $Z_i\in[Y_i]^{<\omega}$ by $Z_i=Y_i-\supp(\vec{a})$. 
  Consider the set 
  \begin{equation*}
   C=\inter_{a\in\fs(\vec{a})}a\blacktriangle A,
  \end{equation*}
  which is an element of $q$ because 
  $\fs(\vec{a})\subseteq B$ and hence it is $\mathscr X$-adequate.
  Therefore we can grab a 
  $(\mathscr Y,2^{\sum_{i<n}|Z_i|}+2m-1)$-witness for the 
  adequacy of $C$, 
  $\langle b_j\big|j<2^{\sum_{i<n}|Z_i|}+2m-1\rangle$. 
  Associate to any element $x\in\inter_{i<n}\fs(Y_i)$ the 
  vector $\langle Z_i\cap Y_i-\supp(x)\big|i<n\rangle$, and 
  notice that there are exactly $2^{\sum_{i<n}|Z_i|}$ many possible 
  distinct such vectors. Thus there exist $2m$ distinct numbers 
  $k_0,\ldots,k_{2m-1}<2^{\sum_{i<n}|Z_i|}+2m-1$ such that for 
  each $j<m$, the vector 
  associated to $b_{k_{2j}}$ is exactly the same as the one  
  associated to $b_{k_{2j+1}}$, and so if we let 
  $c_j=b_{k_{2j}}\bigtriangleup b_{k_{2j+1}}$, then for 
  each $i<n$, $c_j\in\fs(Y_i\setminus Z_i)$. By 
  Lemma~\ref{condensatestigos}, the 
  $m$-sequence $\vec{c}=\langle c_j\big|j<m\rangle$ will be 
  an $m$-witness for the adequacy of $C$. Now let 
  $\vec{d}=\langle d_j\big|j<m\rangle$ 
  be given by $d_j=a_j\bigtriangleup c_j$. We claim that $\vec{d}$ is a 
  $(\mathscr Y,m)$-witness for the adequacy of $A$, so let us fix $i<n$ 
  and let us 
  verify that $\vec{d}$ satisfies conditions~(\ref{adequafsin}) 
  and~(\ref{adequawitness}) from Definition~\ref{defadequacy}. It is 
  certainly the case that 
  $\fs(\vec{c})\subseteq A\cap\fs(Y_i)$, because if $d\in\fs(\vec{d})$ 
  then there are $a\in\fs(\vec{a})$ and $c\in\fs(\vec{c})$ such that 
  $d=a\bigtriangleup c$, and since $c\in C\subseteq a\blacktriangle A$, 
  we get that $d\in A$. Thus requirement~(\ref{adequafsin}) is 
  satisfied. Now for 
  requirement~(\ref{adequawitness}), just grab the $m$-witness for the 
  suitability of $Y_i$ that works for $\vec{a}$, 
  $\langle y_j\big|j<m\rangle$. We constructed the $c_j$ in such a way that 
  $Y_i-\supp(c_j)\cap Z_i=\varnothing$, while 
  $Y_i-\supp(a_j)\subseteq Z_i$. 
  Hence for each $j<m$, 
  $Y_i-\supp(d_j)\cap Z_i=Y_i-\supp(a_j)$ and so whenever 
  $j<m$, $y_j\in Y_i-\supp(d_j)$, and $y_j\notin Y_i-\supp(d_k)$ for 
  $k\neq j$.
 \end{proof}

 Since $H$ is a closed subset of the compact space $\beta\mathbb B$, 
 then $H$ is compact as well, and since it is a semigroup in its 
 own right, we can apply the so-called Ellis-Numakura lemma 
 \cite[Th. 2.5]{hindmanstrauss} which 
 asserts that every (nonempty) compact right-topological semigroup 
 contains idempotent elements. Hence we can pick an idempotent 
 $q\blacktriangle q=q\in H$. Let 
 $A\in\{A_\alpha,\mathbb B\setminus A_\alpha\}$ be such that 
 $A\in q$. We will use $q$ to carefully construct $Y_\alpha$. 
 Let $X=X_{\gamma_\alpha}$.
 
 \begin{claim}
 There is a $Y$, suitable for $X$, such that:
 \begin{enumerate}[(i)]
  \item $\fs(Y)\subseteq A$, and 
  \item\label{secondbullet} For any finite subfamily 
  $\mathscr Y=\{(X_i,Y_i)\big|i<n\}\subseteq\mathscr X$, for 
  any $m<\omega$ and for any finitely many $\xi_0,\ldots,\xi_k<\alpha$, 
  there is a sequence $\langle a_j\big|j<m\rangle$ of elements 
  of $Y$ that is simultaneously 
  an $m$-witness for the suitability (for $X$) of $Y$ and a 
  $(\mathscr Y,m)$-witness for the adequacy of 
  $\inter_{l\leq k}\fs(Y_{\xi_l})$. In particular, $\vec{a}$ witnesses the 
  $(\mathscr Y\cup\{(X,Y)\},m)$-adequacy of 
  $\left(\inter_{l\leq k}\fs(Y_{\xi_l})\right)\cap\fs(Y)$.
 \end{enumerate}
 \end{claim}
 
 \begin{proof}
  This is the only place where we will actually use the hypothesis 
  that\linebreak $\covm=\mathfrak c$. Since $q$ is an idempotent and 
  $A\in q$, the set 
  $A^\star=\{x\in A\big|x\blacktriangle A\in q\}\in q$ and by 
  \cite[Lemma 4.14]{hindmanstrauss}, for every $x\in A^\star$, 
  $x\blacktriangle A^\star\in q$.
  Let $\mathbb P$ be the partial order 
  consisting of those finite subsets $W\subseteq\fs(X)$ such that 
  $\fs(W)\subseteq A^\star$ and satisfying  
  condition~(\ref{elcachito}) from the Definition~\ref{defsuitability} of 
  suitability for $X$, ordered by reverse inclusion (thus $Z\leq W$ means 
  that $Z\supseteq W$). This is a countable forcing notion, hence forcing 
  equivalent to Cohen's forcing. For any finite $\mathscr Y\subseteq\mathscr X$, 
  every $m<\omega$, and all 
  $\xi_0,\ldots,\xi_k<\alpha$ as in part~(\ref{secondbullet}) of the conclusions 
  of this claim, we let 
  $D(\mathscr Y,m,\xi_0,\ldots,\xi_k)$ be the set 
  consisting of all conditions $Z\in\mathbb P$ such that there is 
  an $m$-sequence $\vec{a}$ of elements of $Z$ that simultaneously 
  witnesses the suitability of $Z$ for $X$ and the $(\mathscr Y,m)$-adequacy of 
  $\inter_{l\leq k}\fs(Y_{\xi_l})$. The heart of this proof 
  will be the argument that all these sets 
  $D(\mathscr Y,m,\xi_0,\ldots,\xi_k)$ are dense in $\mathbb P$. 
  Once we have that, we just need to notice that there are $|\alpha|<\mathfrak c=\covm$ 
  many such dense sets, so we 
  can pick a filter $G$ intersecting them all, and we will clearly be 
  done by defining $Y=\union G$.
  
  So let us prove that $D(\mathscr Y,m,\xi_0,\ldots,\xi_k)$ is 
  dense in $\mathbb P$. The idea 
  is that we are given a condition $Z\in\mathbb P$, and we would like to pick a 
  $(\mathscr Y,m)$-witness $\vec{a}$ for the adequacy of 
  $\inter_{l\leq k}\fs(Y_{\xi_l})$, and extend $Z$ to a stronger 
  condition $W$ by adding the range of $\vec{a}$ to it. The main difficulty is 
  that we want $\vec{a}$ to be at the same time an $m$-witness for suitability 
  (for $X$) such that the resulting condition $W=Z\cup\{a_j\big|j<m\}$ still satisfies 
  condition~(\ref{elcachito}) of Definition~\ref{defsuitability}.
  
  Let us start with a condition $Z\in\mathbb P$, and let 
  $X'=X\setminus X-\supp(Z)$. Notice first that we 
  must have $\fs(X')\in q$, for otherwise we would have 
  $\{w\in\fs(X)\big|X-\supp(w)\cap X-\supp(Z)\neq\varnothing\}\in q$, but it is 
  easy to see (arguing as in \cite[Lemma 2.2 and Cor. 2.3]{yonilaboriel}) that this set cannot 
  contain any $\fs$-set, which it should 
  if it was to belong to any idempotent (because of \cite[Th. 5.8]{hindmanstrauss}). Let 
  \begin{equation*}
   B=\left(\inter_{l\leq k}\fs(Y_{\xi_l})\right)\cap\fs(X')\cap\left(\inter_{z\in\fs(Z)}z\blacktriangle A^\star\right).
  \end{equation*}
  Then $B^\star=\{x\in B\big|x\blacktriangle B\in q\}\in q$, thus $B^\star$ is 
  $\mathscr X$-adequate, so there is a 
  $(\mathscr Y,m)$-witness $\vec{a}=\langle a_j\big|j<m\rangle$ for the 
  adequacy of $B^\star$. We will now recursively construct an 
  $m+\binom{m}{2}$-sequence of elements 
  $\vec{x}=\langle x_k\big|k<m+\binom{m}{2}\rangle$ such that 
  $\fs(\vec{x})\subseteq\inter_{a\in\fs(\vec{a})}a\blacktriangle B^\star$ and such 
  that the $X$-supports of its elements are pairwise disjoint and also 
  disjoint from 
  $X-\supp(\vec{a})$, and 
  whose $Y_i$-supports are disjoint from $Y_i-\supp(\vec{a})$ for each 
  $i<n$. If we succeed in this construction, picking a bijection 
  $f:[m]^2\longrightarrow(m+\binom{m}{2})\setminus m$ will enable us to define 
  the sequence $\vec{b}=\langle b_j\big|j<m\rangle$ by:
  \begin{equation*}
   b_j=a_j\bigtriangleup x_j\bigtriangleup\left(\sum_{\substack{k<m \\ k\neq j}}x_{f(\{j,k\})}\right).
  \end{equation*}
  Since the $Y_i$-supports of all the $x_k$ are disjoint from 
  $Y_i-\supp(\vec{a})$, then arguing as in the proof of 
  Claim~\ref{subsemigroup} we conclude that $\vec{b}$ is a 
  $(\mathscr Y,m)$-witness for the adequacy of $B^\star$, hence also 
  for the adequacy of $\inter_{l\leq k}\fs(Y_{\xi_l})$. And the careful 
  choice of the $X$-supports of the $x_k$ ensures that $\vec{b}$ is at the 
  same time an $m$-witness 
  for suitability for $X$, hence letting 
  $W=Z\cup\{b_j\big|j<m\}$ yields a condition in $\mathbb P$ (i.e. 
  $W$ satisfies condition~(\ref{elcachito}) of 
  Definition~\ref{defsuitability}).
  
  Thus, the only remaining issue is that of picking the $x_k$. Assume 
  that we have picked $x_l$ for $l<k$, and we will show how to pick 
  $x_k$. Since $q$ is an idempotent and 
  \begin{equation*}
   C=\inter_{a\in\fs(\vec{a}\frown\langle x_l\big|l<k\rangle)}a\blacktriangle B^\star\in q,
  \end{equation*}
  then there is a set of the form $\fs(V)\subseteq C$ (as before, this follows from 
  \cite[Th. 5.8]{hindmanstrauss}). As in the argument for the proof of 
  Claim~\ref{subsemigroup}, to each element $x\in C$ we associate the 
  vector 
  \begin{multline*}
   \langle Y_i-\supp(\vec{a})\cap Y_i-\supp(x)\big|i<n\rangle\frown \\
   \langle X-\supp(\{a_j\big|j<m\}\cup\{x_l\big|l<k\})\cap X-\supp(x)\rangle,
  \end{multline*}
   and notice that, since there are only finitely many possible distinct such 
  vectors, the infinite set $V$ must contain at least one pair of 
  distinct elements $v,w$ that have the same associated vector. 
  Hence by letting $x_k=v\bigtriangleup w\in\fs(V)\subseteq C$, we get that 
  $Y_i-\supp(x_k)\cap Y_i-\supp(\vec{a})=\varnothing$ for all $i<n$, and 
  $X-\supp(x_k)\cap X-\supp(\{a_j\big|j<m\}\cup\{x_l\big|l<k\})=\varnothing$, so the 
  construction can go on and we are done.
 \end{proof}

 Let $Y_\alpha=Y$. Obviously requirement~(\ref{suitability}) is 
 satisfied, and since $\fs(Y_\alpha)\subseteq A\in\{A_\alpha,\mathbb B\setminus A_\alpha\}$, 
 requirement~(\ref{ultra}) is 
 satisfied as well. It is easy to see that condition~(\ref{secondbullet}) from the 
 conclusion of the claim ensures at once that requirements~(\ref{filterness}) 
 and~(\ref{adequacy}) are fulfilled, and we are done.
\end{proof}

\proof[Acknowledgements]

The results from this paper constitute a portion of the author's PhD Dissertation. 
The author is grateful to his supervisor Juris Stepr\=ans for his encouragement and useful 
suggestions, to the Consejo Nacional 
de Ciencia y Tecnolog\'{\i}a (Conacyt), Mexico, for their financial support via 
scholarship number 213921/309058, as well as to the anonymous referee for her thorough 
reading of the manuscript and for simplifying a couple of proofs.

\end{document}